\documentclass[12pt]{article}
\usepackage[english]{babel}
\usepackage[normalem]{ulem}
\usepackage[all]{xy}
\usepackage{amsmath,amsthm,booktabs,color,dsfont,epsfig,graphicx,mathrsfs,multirow,scalefnt}
\usepackage{amsfonts, amssymb}

\newtheorem{theo}{Theorem}[section]
\newtheorem{cor}[theo]{Corollary}
\newtheorem{lem}[theo]{Lemma}

\newtheorem{prop}[theo]{Proposition}
\newtheorem{defn}[theo]{Definition}
\newtheorem{rmk}[theo]{Remark}
\newtheorem{ex}[theo]{Example}

\newcommand{\Z}{\mathbb{Z}}
\newcommand{\N}{\mathbb{N}}

\newcommand{\s}{\sigma}
\newcommand{\LA}{L_\mathbf{\Lambda}}
\newcommand{\LAA}{L_\mathbf{\Lambda'}}
\newcommand{\LAG}{L_\mathbf{\Gamma}}

\newcommand{\F}{\mathcal{F}}

\newcommand{\Pp}{\mathcal{P}}

\title{Two-sided shift spaces over infinite alphabets}

\author{
\small{Daniel Gon\c{c}alves}\\
\footnotesize{UFSC -- Department of Mathematics}\\
\footnotesize{88040-900 Florian\'{o}polis - SC, Brazil}\\
\footnotesize{\texttt{daemig@gmail.com}}
\and
\small{Marcelo Sobottka}\\
\footnotesize{UFSC -- Department of Mathematics}\\
\footnotesize{88040-900 Florian\'{o}polis - SC, Brazil}\\
\footnotesize{\texttt{sobottka@mtm.ufsc.br}}
\and
\small{Charles Starling}\\
\footnotesize{University of Ottawa}\\
\footnotesize{Dept. of Mathematics and Statistics}\\
\footnotesize{585 King Edward, Ottawa, ON}\\
\footnotesize{K1N 6N5}\\
\footnotesize{\texttt{cstar050@uottawa.ca}}
}

\date{}

\begin{document}

\maketitle

\begin{abstract} Ott, Tomforde, and Willis proposed a useful compactification for one-sided shifts over infinite alphabets. Building from their idea we develop a notion of two-sided shift spaces over infinite alphabets, with an eye towards generalizing a result of Kitchens. As with the one-sided shifts over infinite alphabets our shift spaces are compact Hausdorff spaces but, in contrast to the one-sided setting, our shift map is continuous everywhere. We show that many of the classical results from symbolic dynamics are still true for our two-sided shift spaces. In particular, while for one-sided shifts the problem about whether or not any $M$-step shift is conjugate to an edge shift space is open, for two-sided shifts we can give a positive answer for this question.
\end{abstract}

\bigskip
\hrule
\noindent
{\footnotesize\em This is a pre-copy-editing, author-produced PDF of an article accepted for publication in Journal of the Australian Mathematical Society, following peer review.}
\hrule
\bigskip


\section{Introduction}

This paper is a continuation of a line of study of symbolic dynamics over infinite alphabets initiated in \cite{OTW14} and further developed in \cite{GR, GR1, GoncalvesSobottkaStarling2015,  GSS1, GS1}. In classical symbolic dynamics one starts with a finite alphabet $A$ and constructs the infinite products $A^\Z$ and $A^\N$. These are compact spaces when given the product topology, and the map $\sigma$ which shifts all the entries of the sequences one to the left, is a continuous map. In the case of $A^\Z$ this map is a homeomorphism. A {\em shift space} (or a {\em subshift}) is then a closed subspace of $A^\Z$ or $A^\N$ which is invariant under $\s$. These dynamical systems are fundamental and well-studied -- see \cite{LindMarcus} for an excellent reference.

Given the importance of the above situation, it is natural to wonder what happens if we do not insist that the alphabet $A$ is finite. There has been much research into shift spaces over infinite alphabets, most of it on the case of countable-state shifts of finite type \cite{FF02, FiebigFiebig2005, kitchens1997}. Exel and Laca \cite{EL99} define a C*-algebra $\mathcal{O}_B$ from a given countable $\{0,1\}$-matrix $B$ which is thought of as the incidence matrix of an infinite graph, and propose that the spectrum of a certain commutative C*-subalgebra of $\mathcal{O}_B$ is a good candidate for the Markov shift associated to the graph (see also \cite{Wag88} for a C*-algebra constructed from such a matrix). In \cite{BBG06, BBG07}, Boyle, Buzzi and G\'omez study almost isomorphism for countable-state Markov shifts. Thermodynamical formalism for such shifts has also been well-developed, see for example \cite{CS09, GS98, IY12, MU01, Sa99, Pe14}. Seminal work of Young \cite{Yo98} defines Markov towers and uses them to study hyperbolic systems; see also \cite{Yo99} and see \cite{SV04} for a survey.

A difficulty one encounters with infinite alphabet shifts is that if $A$ is infinite then the infinite products $A^\N$ and $A^\Z$ are not compact, and indeed not even locally compact. This difficulty is dealt with in various ways in the papers listed above based on the application at hand. Our interest is in a compactification of $A^\N$ put forth by Ott, Tomforde, and Willis in \cite{OTW14} -- this compactification can be identified with the set of all finite sequences (including an empty sequence) and infinite sequences over $A$, and was motivated by spaces arising from countable graphs \cite{We14}. From this they say that subshifts are closed (hence compact) subsets $\Lambda\subset \Sigma^\N_A$ invariant under the shift and with an additional condition which guarantees that the set of infinite sequences in $\Lambda$ is dense in $\Lambda$. As noted in the introduction of \cite{OTW14}, compactness of their subshifts is an important feature that allows them to establish results analogous to classical results about finite alphabet subshifts. To us, the space constructed in \cite{OTW14} seems so natural that it deserves study in its own right. Further to this, in \cite{GoncalvesSobottkaStarling2015} we defined sliding block codes on such spaces and studied to what extent such maps characterize shift-commuting continuous maps between Ott-Tomforde-Willis shifts, and there compactness played a central role. See also \cite{GR}, where a one-sided $(M+1)$-shift which is not conjugate (in the sense of \cite{OTW14}) to a one-sided $M$-step shift is constructed.

For the work at hand, our main motivation was to generalize the paper of Kitchens \cite{kitchens} to the Ott-Tomforde-Willis infinite alphabet case. In fact, our work on generalizing Kitchens' work grew to lead to not only the present work but also \cite{GSS1} and the previously mentioned \cite{GoncalvesSobottkaStarling2015}. We briefly describe Kitchens' result. If $A$ is a \underline{finite} group, then the two-sided full shift $A^\Z$ is a compact zero-dimensional topological group with the operation of pointwise multiplication, and the shift map is an expansive group automorphism. Kitchens proves a converse to this: \cite[Theorem 1.i]{kitchens} says that if $X$ is a compact zero-dimensional topological group and $\phi$ is an expansive group automorphism of $X$, then $(X,\phi)$ is topologically conjugate to a {\em subshift} over a finite group via a group isomorphism. He also proves \cite[Theorem 1.ii]{kitchens} that given the same assumptions, $(X,\phi)$ is topologically conjugate to $(A^\Z, \sigma)\times (F, \tau)$, where $F$ is a finite group and $\tau$ is an automorphism. That the topology is compact and zero dimensional, and that the shift is bijective, are crucial in the proofs.

Hence, our goal is to define and study a two-sided compact analogue of the Ott-Tomforde-Willis construction. Our first clue of how to accomplish this comes from the classical situation, where one can always obtain a two-sided shift from a one-sided shift via the inverse limit construction. Applying this idea to Ott-Tomforde-Willis shifts yields a set which can be identified with the set, which we denote $\Sigma^\Z_A$, of all infinite sequences in $A$, all left-infinite sequences in $A$ ending at an arbitrary integer index, and an empty sequence. This is the set which underlies the topological space we define, see Definition \ref{2sideddef}. Since the shift map is not continuous for Ott-Tomforde-Willis shifts, we do not take the topology to be the product topology (as is usual for topological inverse limits) and instead take a suitably modified topology generated by cylinders corresponding to left-infinite words, see Lemma \ref{topologylemma}. This topology makes $\Sigma^\Z_A$ zero-dimensional, compact, and Hausdorff, and in this topology the shift map is a homeomorphism. Using this topology we were successfully able to obtain results akin to Kitchens' \cite[Theorem 1.ii]{kitchens} in \cite[Theorem 5.18]{GSS1} (although we note that we were only able to generalize \cite[Theorem 1.i]{kitchens} for the one-sided case, see \cite[Propositions 3.2 and 3.5]{GSS1}).

In this paper, we focus on the fundamental properties of $\Sigma^\Z_A$ and its subshifts. Namely, we prove that any shift space is determined by a set of forbidden words (Proposition \ref{shiftforbiddenwords}), and determine to what extent the shift-commuting continuous maps between shift spaces are given by sliding block codes (Theorem \ref{continuoussliding block code}). We also show that our shift spaces are conjugate to their higher-block presentations (Proposition \ref{M_Higher_Block_Code}) which implies that every $M$-step shift is conjugate to the edge shift on some graph (Proposition \ref{edge_shift_conjugacy}). Finally, we show that one can go back and forth from one-sided shifts to two-sided shifts via projection and (set-theoretical) inverse limit (Propositions \ref{projection} and \ref{inverse_limit2}), though these processes are not in general inverses of each other.

The article is organized as follows. In Section \ref{SEC:2-sided_shift_spaces} we recall background and definitions from \cite{OTW14} and define our notion of the two-sided full shift over an infinite alphabet, including a careful description of the topology. We then define our notion of a shift space in analogy to that in \cite{OTW14} and reproduce fundamental results from finite alphabet symbolic dynamics, adapted for the infinite-alphabet situation.
In Section \ref{SEC:SBC} we discuss sliding block codes between our shift spaces and discuss their higher-block presentations. In Section \ref{SEC:2-sided_1-sided} we present the relationship between the two-sided shift spaces defined here and Ott-Tomforde-Willis one-sided shift spaces. In Section \ref{SEC:final} we provide a final discussion.

\section{Two-sided shift spaces over infinite alphabets}\label{SEC:2-sided_shift_spaces}

\subsection{The Ott-Tomforde-Willis one-sided full shift}\label{SEC:OTW}
We briefly recall the construction of \cite{OTW14} of the one-sided full shift over an infinite alphabet. This topology is inspired by spaces associated to infinite graphs \cite{We14}, and one sees similar topologies associated to semilattices arising from C*-algebras \cite{BdCM15,Ex08, Pa02}.

Let $A$ be a countably infinite alphabet, and define a new symbol $\o$ not in $A$. We call  $\o$ the {\em empty letter}, and let $\tilde A:=A\cup\{\o\}$. Let
\[
\Sigma_A^\N = \{(x_i)_{i\in \N}\in \tilde A^\N: x_i = \o \text{ implies }x_{i+1}= \o \},
\]
\[
\Sigma_A^{\N\ \text{fin}} = A^\N,\hspace{1cm}\Sigma_A^{\N\ \text{inf}} = \Sigma_A^\N\setminus\Sigma_A^{\N\ \text{inf}}.
\]
The set $\Sigma_A^{\N\ \text{fin}}$ is identified with the finite sequences in $A$ via the identification
\[
(x_1x_2\dots x_k\o\o\o\dots)\equiv (x_1x_2\ldots x_k)
\]
and the sequence of all $\o$s is denoted $\O$ and called the {\em empty sequence}. For a finite set $F\subset A$ and $x\in \Sigma_A^{\N\ \text{fin}}$, let
\[
Z(x,F):=\{y\in\Sigma^\N_A:\ y_i=x_i\ \forall 1\leq i\leq k,\ y_{k+1}\notin F\}.
\]
Sets of this form are called {\em generalized cylinders}. Endowed with the topology generated by the generalized cylinders, $\Sigma_A^\N$ is a compact totally disconnected metrizable Hausdorff space, and in this topology the generalized cylinders are compact and open \cite[Proposition 2.5, Theorem 2.15, Corollary 2.18]{OTW14}. We note that the shift map (the map which shifts every entry one to the left) is continuous everywhere except at $\O$ \cite[Proposition 2.23]{OTW14}.


\subsection{Construction of the two-sided shift -- its topology and dynamics}
As mentioned in the introduction, we are interested in the two-sided analogue of the above. In the finite alphabet case, the two-sided full shift is the inverse limit of the one-sided full shift -- we take this as our starting point. We first write down the set we obtain from the inverse limit, and then describe the topology we endow this set with. After we do this, in Remark \ref{inverse_limit1} we discuss the relation with the inverse limit.

We take $A$ and $\tilde A$ as in Section \ref{SEC:OTW}. Consider the subsets $\Sigma_A^{\Z\ \text{inf}},\Sigma_A^{\Z\ \text{fin}}\subset\tilde A^\Z$ given by
$$\Sigma_A^{\Z\ \text{inf}}:=A^\Z:=\{(x_i)_{i\in\Z}:\ x_i\in A\ \forall i\in \Z\}$$ and $$\Sigma_A^{\Z\ \text{fin}}:=\{(x_i)_{i\in\Z}\in \tilde A^\Z: x_k = \o \text{ for some }k\in \Z\text{ and }x_i = \o \text{ implies } x_{i+1} = \o \text{ for all }i\in \Z\}.$$

\begin{defn}\label{2sideddef}
Let $A$ be a countably infinite set and let $\Sigma_A^{\Z\ \text{inf}}$ and $\Sigma_A^{\Z\ \text{fin}}$ be as above. The {\em two-sided full shift} over $A$ is the set $$\Sigma_A^\Z:=\Sigma_A^{\Z\ \text{inf}}\cup \Sigma_A^{\Z\ \text{fin}}.$$ Also, given $x\in\Sigma_A^\Z$, define the length of $x$ as $l(x):=\sup\{i:x_i\neq\o\}$.
\end{defn}

We will refer to the constant sequence $\O:=(\ldots\o\o\o \ldots)\in \Sigma_A^{\Z\ \text{fin}}$ as the {\em empty sequence} -- we note that $l(\O)=-\infty$. The elements of $\Sigma_A^{\Z\ \text{fin}}$ will be referred to as {\em finite sequences} and we will identify $\Sigma_A^{\Z\ \text{fin}}$ with the set $\{\O\}\cup\bigcup_{i\in \Z} \prod_{-\infty}^{i} A$ via the identification $$(\ldots x_ix_{i+1}x_{i+2}\ldots x_k\o\o\o\o\ldots)\equiv (\ldots x_ix_{i+1}x_{i+2}\ldots x_k).$$ Following this identification, we will use the notation $x=(x_i)_{i\leq k}$ to refer to a point in $\Sigma_A^{\Z\ \text{fin}}$ with length $-\infty<k < \infty$. If needed, we will separate the 0th entry from the first entry by a period; i.e., if $k >0$,
\[
(x_i)_{i\leq k} = (\dots x_{-2}x_{-1}x_0.x_1x_2\dots x_k)
\]

We now define a topology on $\Sigma_A^\Z$ which will have a basis of clopen sets.
First, we notice that the topology of generalized cylinders defined for Ott-Tomforde-Willis one-sided shifts (see \cite[Definition 2.8]{OTW14}) cannot be straightforwardly adapted to our case, since cylinders which are defined using only finitely many entries will not be closed sets. To deal with this we define generalized cylinder sets of $\Sigma_A^\Z$ as sets obtained by specifying infinitely many coordinates to the left. This definition is aimed towards making our space manageable while keeping its geometric origins intact.\footnote{This type of approach is not uncommon in the literature, see for example \cite[Definition 3.2]{Put10}.} We are more precise below.

\begin{defn} Let $x=(x_i)_{i\leq k}\in \Sigma_A^{\Z\ \text{fin}}$, $x\neq \O$, and let $F\subset A$ be a finite set. Define $$Z(x,F):=\{y\in\Sigma^\Z_A:\ y_i=x_i\ \forall i\leq k,\ y_{k+1}\notin F\}.$$
If $F$ is the empty set we will shorten $Z(x,F)$ to $Z(x)$.
\end{defn}

Let $B_{\Sigma^\Z_A}$ be the collection of all generalized cylinders $Z(x,F)$, together with sets that are complements of finite unions of generalized cylinders of the form $Z(x)$. We endow $\Sigma_A^\Z$ with the topology $\tau_{\Sigma^\Z_A}$ which is generated by $B_{\Sigma^\Z_A}$.

\begin{lem}\label{topologylemma} The topology $\tau_{\Sigma^\Z_A}$ is Hausdorff and $B_{\Sigma^\Z_A}$ is a clopen basis for it.
\end{lem}

\begin{proof}
We first show that all elements of $B_{\Sigma^\Z_A}$ are clopen in $\tau_{\Sigma^\Z_A}$. Since $B_{\Sigma^\Z_A}$ generates $\tau_{\Sigma^\Z_A}$, generalized cylinders of the form $Z(x)$ and sets that are complements of finite unions of generalized cylinders of the form $Z(x)$ are clopen. Furthermore, generalized cylinders of the form $Z(x,F)$ are open sets by assumption. Thus, what remains to be shown is that sets of the form $Z(x,F)$ are closed. This easily follows by noticing that, for $x=(\ldots,x_{k-1},x_k)\in\Sigma_A^{\Z\ \text{fin}}$ and finite $F\subset A$, we have that
$$Z(x,F)=Z(\ldots,x_{k-1},x_k)\cap\left[\bigcup_{f\in F} Z(\ldots,x_{k-1},x_k,f)\right]^c,$$ a finite intersection of closed sets.\\

Let us show that $B_{\Sigma^\Z_A}$ is a basis for $\tau_{\Sigma^\Z_A}$. Note that given any $x\in\Sigma^\Z_A$, $x\neq \O$, and any $k\leq l(x)$, we have that $x$ belongs to $Z(\ldots,x_{k-1},x_k)$. On the other hand, $\O$ belongs to the sets that are complements of finite unions of generalized cylinders of the form $Z(x)$. Now, let $x,y,x^1,\ldots,x^p,y^1,\ldots,y^q\in\Sigma_A^{\Z\ \text{fin}}\setminus\{\O\}$ and let $F,G\subset A$ be finite sets. We have that,
\begin{equation}\label{basis1}\left[\bigcup_{i=1}^p Z(x^i)\right]^c\cap\left[\bigcup_{i=1}^q Z(y^i)\right]^c=\left[\bigcup_{i=1}^p Z(x^i)\cup\bigcup_{i=1}^q Z(y^i)\right]^c,\end{equation}
while, supposing that $l(x)\leq l(y)$, we have that
\begin{equation}\label{basis2}Z(x,F)\cap Z(y,G)=\left\{\begin{array}{lcl}
Z(y,G) &,&\text{if } l(x)<l(y),\ x_i=y_i\ \forall i\leq l(x)\text{ and } y_{l(x)+1}\notin F, \\
Z(x,F\cup G) &,&\text{if } l(x)=l(y)\text{ and } x_i=y_i,\ \forall i\leq l(x),\\
\emptyset &,&\text{otherwise.}
\end{array}\right.\end{equation}
Thus, $\left[\bigcup_{i=1}^p Z(x^i)\right]^c\cap\left[\bigcup_{i=1}^q Z(y^i)\right]^c\in B_{\Sigma^\Z_A}$, while $Z(x,F)\cap Z(y,G)\in B_{\Sigma^\Z_A}$ whenever it is not empty.

On the other hand,
\begin{equation}\label{basis3} Z(x,F)\cap \left[\bigcup_{i=1}^q Z(y^i)\right]^c = Z(x,F)\cap \bigcap_{i=1}^q  Z(y^i)^c,\end{equation}
and therefore, if the above set is non-empty, then $Z(x,F)\cap Z(y^i)^c \neq \emptyset$, for all $i$. Note that, if $l(y^i) \leq l(x)$ then $Z(x,F)\cap Z(y^i)^c$ is either $\emptyset$ or $Z(x,F)$. Furthermore, if $l(y^i)>l(x)$ then $Z(x,F)\cap Z(y^i)=\emptyset$ implies that $Z(x,F)\cap Z(y^i)^c = Z(x,F)$, while if $Z(x,F)\cap Z(y^i)\neq \emptyset$ then $y^i_j =x_j$ for all $j\leq l(x)$ and $y^i_{l(x)+1} \notin F$. Let $i_1,i_2,\ldots,i_r$ be all the indices such that $l(y^i)>l(x)$ and $Z(x,F)\cap Z(y^i)\neq \emptyset$, taken so that $l(y^{i_\ell})\leq l(y^{i_{\ell+1}})$. Then we can rewrite \eqref{basis3} as
\begin{equation}\label{basis3.1} Z(x,F)\cap \left[\bigcup_{i=1}^q Z(y^i)\right]^c = Z(x,F)\cap \bigcap_{\ell=1}^r  Z(y^{i_\ell})^c.\end{equation}

Now, let $z\in Z(x,F)\cap \left[\bigcup_{i=1}^r Z(y^{i_\ell})\right]^c$. If $l(y^{i_\ell})>l(z)$, for all $\ell=1\ldots r$, then let $H:=\left\{y^{i_\ell}_{l(z)+1}:\ l=1 \ldots r \right\}$ and notice that $z \in Z((z_i)_{i\leq _{l(z)}},F\cup H)\subseteq Z(x,F)\cap \left[\bigcup_{i=1}^r Z(y^{i_\ell})\right]^c$. If $l(y^{i_\ell})\leq l(z)$, for some $\ell$, then let $L:=\max\{\ell:\  l(y^{i_\ell})\leq l(z) \}$. We have that
$$ \begin{array}{lcl}
z_j=x_j, &\text{for}& j\leq l(x),\\\\
z_{l(x)+1}\notin F,\\\\
(z_j)_{j=l(x)+1,\ldots, l(y^{i_\ell})}\neq (y^{i_\ell}_j)_{j=l(x)+1,\ldots, l(y^{i_\ell})},&\text{for all}& \ell\leq L.
\end{array}
$$
Hence, setting $H:=\left\{y^{i_\ell}_{l(y^{i_L})+1}:\ \ell>L\right\}$, we have that
\begin{equation}\label{basis3.2}z\in Z(\ldots,z_{l(x)},\ldots, z_{l(y^{i_L})}, H)\subset Z(x,F)\cap \left[\bigcup_{i=1}^q Z(y^i)\right]^c.\end{equation}\\

Finally, if $x, y\in \Sigma^\Z_A$ are distinct points they differ at some entry, so we can clearly separate them with elements of $B_{\Sigma^\Z_A}$. Hence this topology is Hausdorff.

\end{proof}

Calculations in the proof of the above lemma imply the following, which is analogous to  \cite[Theorem 2.16]{OTW14}.
\begin{lem}\label{nhoodbasis}
Suppose that $x\in \Sigma_A^\Z$, $x\neq\O$. If $l(x) = \infty$ then the family of sets
\[
\{ Z(...,x_{n-1},x_{n}):\ n\in \Z\} \subset B_{\Sigma_A^\Z}
\]
is a neighborhood basis for $x$. If $l(x)< \infty$ then the family of sets
\[
\{Z(x, F):\ F\text{ is a finite subset of }A\}  \subset B_{\Sigma_A^\Z}
\]
is a neighborhood basis for $x$.
\end{lem}

We note that while for one-sided shifts the topology of generalized cylinders
coincides  on $\Sigma_A^{\N\ \text{inf}}$ with the product topology, the same does not hold for two-sided shifts. In fact, our topology on $\Sigma_A^{\Z\ \text{inf}}$ is strictly finer than the product topology.

\begin{rmk}\label{inverse_limit1}
The set $\Sigma_A^\Z$ can be identified with the inverse-limit dynamical system of the Ott-Tomforde-Willis one-sided shift $\Sigma_A^\N$, $$(\Sigma_A^\N)^\s:=\{(\mathcal{X}_i)_{i\in\Z}:\ \forall i\in\Z\ \mathcal{X}_i\in\Sigma_A^\N\text{ and }\s(\mathcal{X}_i)=\mathcal{X}_{i+1}\},$$
via the map $p:(\Sigma_A^\N)^\s\to\Sigma_A^\Z$ given by $p\big((\mathcal{X}_i)_{i\in\Z}\big)=(x_i)_{i\in\Z}$, where $x_i$ is the first symbol of the sequence $\mathcal{X}_i$. In fact, $p$ is a bijection whose inverse is given by $p^{-1}\big((x_i)_{i\in\Z}\big)= (\mathcal{X}_i)_{i\in\Z}$, where $\mathcal{X}_i=(x_{i+j-1})_{j\in\N}$.

However, note that this is only a bijection and not a homeomorphism, since the natural topology considered for inverse-limit dynamical systems is the product topology defined from the topology of $\Sigma_A^\N$.
\end{rmk}

While we have seen that our topology is Hausdorff, the next proposition shows that $\Sigma_A^\Z$ is not metrizable.

\begin{prop}\label{non1countable} Given the topology above, $\Sigma_A^\Z$ is not first countable.
\end{prop}

\begin{proof}
Suppose $\{V_i\}$ is a countable basis at $\O$. Then for each $i$ there exists an $U_i:=\displaystyle \left[\bigcup_{1\leq j\leq n_i}Z(x^j)\right]^c$ such that $\O \in U_i \subseteq V_i$. Since $\Sigma_A^\Z$ is Hausdorff we have that $\displaystyle \bigcap_{i=1}^{\infty}U_i = \{\O\}$, so $\displaystyle \bigcup_{i=1}^{\infty}U_i^c = \Sigma_A^\Z \setminus \{\O\}$ and hence $\Sigma_A^\Z \setminus \{\O\}$ can be written as a countable union of cylinder sets, say $\displaystyle \Sigma_A^\Z \setminus \{\O\} = \bigcup_{i=1}^{\infty} Z(x^i)$.

Take $y_{l(x^1)}$ to be different from  $(x^1)_{l(x^1)}$, let $k_1 = \min \{l(x^1)-1, l(x^2)\}$, and take $y_{k_1} \neq (x^2)_{k_1}$. For each $n\geq 1$, suppose $y_{k_n}$ is defined, let $k_{n+1} = \min \{ k_n -1, l(x^{n+1})\}$ and take $y_{k_{n+1}}\neq (x^{n+1})_{k_{n+1}}$. For each $n\in \Z$ not of the form $k_i$ for some $i$, let $y_n$ to be equal to a constant letter different from $\o$. It follows that $y\in\Sigma_A^{\Z\ \text{inf}}$, but $y$ is not in any cylinder set $Z(x^i)$. We have a contradiction and so $\Sigma_A^\Z$ is not first countable.

\end{proof}

The above proof uses the point $\O$ to prove that first countability fails. We now show that if we remove this point, the space becomes first countable, but neither second countable nor separable.

\begin{prop}\label{non2countable} The space $\Sigma_A^\Z\setminus\{\O\}$ is first countable, but it is neither second countable nor separable.
\end{prop}

\begin{proof}
$\Sigma_A^\Z\setminus\{\O\}$ is first countable since Lemma \ref{nhoodbasis} gives a countable neighborhood basis for each point of $\Sigma_A^\Z\setminus\{\O\}$. To see that $\Sigma_A^\Z\setminus\{\O\}$ is neither second countable nor separable, notice that the collection of open sets $\{Z((x_i)_{i\leq 0}): (x_i) \text{ is a sequence of length zero}\}$ is pairwise disjoint and uncountable.
	


\end{proof}

\begin{rmk}\label{convergence_to_O}Since there is no countable neighbourhood basis for $\O$, sequences are not adequate to describe continuity at $\O$ (though they can be used to show that a map is not continuous at $\O$). Moreover, convergence of a sequence to $\O$ can be tricky. For example, a sequence $(x^j)$ such that $l(x^j) \rightarrow -\infty$ converges to $\O$. Also sequences where, for all $j\in\N$ and $n\in\Z$, there exists $k>j$ and  $i<n$ such that $x_i^k\neq x_i^j$ will converge to $\O$. Indeed, take $a,b\in A$ and for all $j\in \N$ let $x^j=(x_i^j)_{i\in\Z}\in\Sigma_A^\Z$ defined as $x_i^j=a$ for all $i\neq -j$ and $x_{-j}^j=b$ -- the sequence $(x^j)$ converges to $\O$.

More generally, a sequence $(x^j)$ converges to $\O$ if and only if for any $Z(y)$ there are only a finite number of points $x^j$ which belong to $Z(y)$.
\end{rmk}

\begin{rmk}
Despite the oddities of our space implied by Propositions \ref{non1countable} and \ref{non2countable}, we include $\O$ in our space for two reasons. Firstly, since the space is built from the inverse limit of $\Sigma_A^\N$, $\O$ is naturally included as the point corresponding to the infinite sequence which is constantly $\O\in \Sigma_A^\N$. Secondly, our main application is to generalize work of Kitchens \cite{kitchens} to shift spaces which have {\em inverse semigroup} operations. Inverse semigroups are typically assumed to have a zero element, and the point $\O$ naturally fills this role \cite{GSS1}.
\end{rmk}

The characterization of the sequences that converge to $\O$ given in Remark \ref{convergence_to_O} above can be used to prove the following result.

\begin{cor}\label{accumulation_at_O} Let $X\subset\Sigma_A^\Z$. Then $\O$ is an accumulation point of $X$ if and only if $X\setminus\{\O\}$ is not contained in a finite union of generalized cylinders.
\end{cor}

The two-sided full shift $\Sigma_A^\Z$ can be projected onto the positive coordinates to yield the one-sided full shift $\Sigma_A^\N$. Let $\pi:\Sigma_A^\Z\to\Sigma_A^\N$ denote this projection map. Then clearly we have that if $x\in\Sigma_A^{\Z}$ and $l(x)\geq 1$ then $l\big(\pi(x)\big)=l(x)$ (where the $l$ in the left hand side stands for the length of a sequence in the one-sided shift).

\begin{prop}\label{k-projection}
Let $\pi:\Sigma_A^\Z\to\Sigma_A^\N$ be as above. Then $\pi$ is continuous at $x\in\Sigma_A^\Z$ if and only if $l(x)\geq 0$.
\end{prop}

\begin{proof}

Let $x=(x_i)_{i\leq l(x)}\in\Sigma_A^\Z$. First suppose that $l(x)\geq 1$. Then $\pi(x)=(x_i)_{1\leq i\leq l(x)}$. Given $k\in\N$, $1\leq k\leq l(x)$, let $Z(x_1x_2\ldots x_k,F)$ be a generalized cylinder of $\Sigma_A^\N$ containing
$\pi(x)$. Then $Z\big((x_i)_{i\leq k},F\big)$ is a generalized cylinder of $\Sigma_A^\Z$ which contains $x$ and such that $\pi\Big(Z\big((x_i)_{i\leq k},F\big)\Big)=Z(x_1x_2\ldots x_k,F)$.
In the case $l(x)=0$, we have that $\pi(x)=\O$ and hence, given a neighborhood $Z(\O,F)$ of the empty sequence in $\Sigma_A^\N$, we have that $Z(x,F)$ is a neighborhood of $x$ in $\Sigma_A^\Z$ such that
$\pi\big(Z(x,F)\big)=Z(\O,F)$. Thus we have that if $l(x) \geq 0$ then $\pi$ is continuous.

Now suppose $-\infty<l(x)\leq -1$. Notice that $\pi(x)=\O$. Find an element $(z^n)_{n\geq 1}\in\Sigma_A^{\Z\ \text{inf}}$, such that $z^n_i=x_i$ for all $i\leq l(x)$, $z^n_{l(x)+1}\neq z^m_{l(x)+1}$ for all $n\neq m$, and
$z^n_i=z^m_i$ for all $i\geq 1$ and $m,n\geq 1$. The sequence $z^n$ converges to $x$, but $\pi(z^n)=(z^n_i)_{i\in\N}=(z^1_i)_{i\in\N}\neq\O = \pi(x)$ and hence $\pi$ is not continuous at $x$.

Finally, suppose $l(x)=-\infty$, that is, $x=\O$. Let $(z^n)_{n\geq 1}\in\Sigma_A^{\Z\ \text{inf}}$ be a sequence such that $z^n_i=z^m_i$ for all $i,m,n \geq 1$ and such that $z^n_{i}\neq z^m_{i}$ for all $i\leq 0$ and $n\neq m \geq 1$. Then $(z^n)$ converges to $\O$, but $\pi(z^n)=(z^n_i)_{i\in\N}=(z^1_i)_{i\in\N}\neq\O$ and hence $\pi$ is not continuous at $\O$.

\end{proof}

The {\em shift map} $\s:\Sigma_A^\Z\to\Sigma_A^\Z$ is given by $\s\big((x_i)_{i\in\Z}\big)= (x_{i+1})_{i\in\Z}$ for all $(x_i)_{i\in\Z}\in \Sigma_A^\Z$. In particular, we have that
$\s(\O)=\O$ and if $x\in\Sigma_A^{\Z\ \text{fin}}\setminus\{\O\}$ then $\s(\ldots x_{-2}x_{-1}x_0.x_1x_2\ldots x_k)=(\ldots x_{-1}x_0x_1.x_2x_3\ldots x_k)$.

Note that the shift map for two-sided shifts is invertible. Furthermore, while the shift map on $\Sigma_A^\N$ is not continuous at $\O$ (see \cite[Proposition 2.23]{OTW14}), we show below that the shift map on $\Sigma_A^\Z$ is continuous everywhere.

\begin{prop}
The shift map $\s:\Sigma_A^\Z\to\Sigma_A^\Z$ is a homeomorphism.
\end{prop}

\begin{proof}

 Since $\Sigma_A^\Z$ is compact and Hausdorff and $\s$ is a bijection, we just need to prove that $\s$ is continuous.

Let $x\in\Sigma_A^{\Z\ \text{fin}}$, with $l(x)=k>-\infty$, and let $F\subset A$ be a finite subset. Since $l(\sigma^{-1}(x))=k+1$, for each cylinder $Z(x,F)$ we have
that $\s^{-1}\big(Z(x,F)\big)=Z(\sigma^{-1}(x),F)$.
Also, given $x^1,\ldots,x^m\in \Sigma_A^{\Z\ \text{fin}}$, let $\mathcal{Z}:=\left[\bigcup_{1\leq j\leq m}Z(x^j)\right]^c$ be neighborhood of $\O$. We calculate
 $$\begin{array}{lcl}\s^{-1}\big(\mathcal{Z}\big)&=&\displaystyle \s^{-1}\left(\left[\bigcup_{1\leq j\leq m}Z(x^j)\right]^c\right)=\s^{-1}\left(\bigcap_{1\leq j\leq m}Z(x^j)^c\right)\\\\&=&\displaystyle \bigcap_{1\leq j\leq
 m}\left[\s^{-1}\left(Z(x^j)\right)\right]^c=\bigcap_{1\leq j\leq m}Z(\s^{-1}(x^j))^c=\left[\bigcup_{1\leq j\leq m}Z(\s^{-1}(x^j))\right]^c\end{array}$$
and hence $\s$ is continuous.

\end{proof}

The following proposition will be useful to us later, and provides another link between our construction and that of \cite{OTW14}.

\begin{prop}\label{compact_cylindres} For all $x \in  \Sigma_A^{\Z\ \text{fin}}\setminus\{\O\}$, the cylinder $Z(x) \subset \Sigma_A^{\Z}$ with the relative topology is homeomorphic to the Ott-Tomforde-Willis full shift $\Sigma_A^{\N}$.
\end{prop}

\begin{proof}
Take $x \in  \Sigma_A^{\Z\ \text{fin}}\setminus\{\O\}$ and define a map $f_x: \Sigma_A^\N \to Z(x)$ by $$\left(f_x(z)\right)_i=\left\{\begin{array}{lcl}
x_i, &\text{ if}& i\leq l(x)\\
z_{i-l(x)}, &\text{ if}& i \geq l(x)+1.
\end{array}\right.
$$ We claim that $f_x$ is a homeomorphism. This map is clearly bijective and, since $ \Sigma_A^\N$ is compact and $Z(x)$ is Hausdorff, our claim will be proven if we show that $f_x$ is continuous. By Lemma \ref{nhoodbasis} every open set in $Z(x)$ is a union of open sets of the form $Z(x)\cap Z(y, G)$, where $y\in \Sigma_A^{\Z\ \text{fin}}$ and $G\subset A$ is finite. A short calculation gives that

\begin{equation*}f_x^{-1}\big(Z(x)\cap Z(y,G)\big)=\begin{cases}
Z(y_{l(x)+1}\ldots y_{l(y)},G) &\text{if } l(x)<l(y)\ \text{and } x_i=y_i\ \forall i\leq l(x),\\
Z(\O,G) &\text{if } l(x)=l(y)\text{ and } x_i=y_i,\ \forall i\leq l(x),\\
\Sigma_A^\N &\text{if }l(x)>l(y), x_i=y_i\ \forall i \leq l(y) \text{ and } x_{l(y)+i}\notin G\\
\emptyset &\text{otherwise.}
\end{cases}
\end{equation*}

In any case, $f_x^{-1}\big(Z(x)\cap Z(y,G)\big)$ is open in $\Sigma_A^\N$, and so $f_x$ is continuous.
\end{proof}
Now, in a similar fashion to \cite{OTW14}, we show that our construction yields a compact space.

\begin{prop} $\Sigma_A^\Z$ is compact and sequentially compact.
\end{prop}

\begin{proof}

Let $\{V_{\alpha}\}$ be an open cover for $\Sigma_A^\Z$. Without loss of generality we may assume that each $V_\alpha$ is either a cylinder or the complement of a finite union of cylinders. Since $\O$ does not belong to any  cylinder set, there exists $\alpha_0$ such that $V_{\alpha_0} = \displaystyle \left[\bigcup_{1\leq j\leq m}Z(x^j)\right]^c$, where $x^j \in  \Sigma_A^{\Z\ \text{fin}}$.
Notice that $\{V_{\alpha}:\alpha \neq \alpha_0 \}$ covers $\bigcup_{1\leq j\leq m}Z(x^j) $, and so compactness of each $Z(x^j)$ (Proposition \ref{compact_cylindres}) implies we can find a finite subset $\{V_{\alpha_i}\}_{i=1}^{K}$ of $\{V_{\alpha}\}$ which covers $\bigcup_{1\leq j\leq m}Z(x^j)$. Hence, $\{V_{\alpha_i}\}_{i=1}^{K}\cup \{V_{\alpha_0}\}$ covers $\Sigma_A^\Z$, and so $\Sigma_A^\Z$ is compact.

To check sequential compactness, let $(y^i)_{i\in\N}$ be any sequence in $\Sigma_A^\Z$. First, suppose that there does not exist any cylinder $Z(x)\subset \Sigma_A^\Z$ such that $y^i\in Z(x)$ for infinitely many $i\in\N$. In such a case, we have that $y^i$ converges to $\O$ as $i$ goes to infinity. On the other hand, if there exists $Z(x)\subset \Sigma_A^\Z$ such that $y^i\in Z(x)$ for infinitely many  $i\in\N$, then we can take a subsequence $(y^{i_\ell})_{\ell\in\N}$ such that $y^{i_\ell}\in Z(x)$ for all $\ell\in\N$. Since $Z(x)$ is homeomorphic to $\Sigma_A^\N$, and $\Sigma_A^\N$ is sequentially compact (because it is a compact metric space), then $Z(x)$ is also sequentially compact and there exists a subsequence of  $(y^{i_\ell})_{\ell\in\N}$ which converges in $Z(x)$.
\end{proof}

\subsection{Two-sided shift spaces}
In this section we define shift spaces in our context. We then show that any given shift space can be characterized by a set of forbidden words. The set of forbidden words ends up containing some left-infinite words -- this seems only natural since our topology also depends on such words.

Given a subset $X\subseteq\Sigma^\Z_A$, let \begin{equation}\label{shift_inf_fin} X^{\text{fin}}:=X\cap\Sigma_A^{\Z\ \text{fin}} \text{ and }
X^{\text{inf}}:=X\cap\Sigma_A^{\Z\ \text{inf}}\end{equation}
be the set of all finite sequences in $X$ and the set of all infinite sequences in $X$, respectively. Given any set $D$ we let $D^{\Z^-}:=\{(x_i)_{i\leq 0}:\ x_i\in D\}$ and let
\begin{equation}\label{left-infty_subblocks}B_{\text{linf}}(X):=\{(a_i)_{i\leq 0}\in (\tilde A)^{\Z^-}:\ \exists \ x\in X,k\in\Z\text{ such that } x_{i+k}=a_i\ \forall i\leq 0\}\end{equation} be the set of left-infinite subblocks of $X$. For $1\leq n<\infty$ let
\begin{equation}\label{subblocks}B_n(X):=\{(x_1,\ldots,x_n)\in {\tilde A}^n: x_1,\ldots,x_n\text{ is a subblock of some sequence in }X\}\end{equation} be the set of all {\em blocks of length $n$} in $X$. Note that blocks in $B_n(X)$ may contain the empty letter.

We single out the blocks of length one, and use the notation $\LA:= B_1(X)\setminus\{\o\}$ -- this is the set of all symbols of $A$ used by sequences of $X$, or the {\em letters} of $X$. The {\em language} of $X$ is \begin{equation}\label{language}B(X):=\bigcup_{n\in\N}B_n(X).\end{equation}



Before we define shift spaces we need the following definitions.

\begin{defn}
Let $X\subset \Sigma^\Z_A$ and $k\geq 1$. Given $1\leq n <\infty$ and $a\in B_n(X)$, the {\em $k^{th}$ follower set} of $a$ in $X$ is the set \begin{equation}\label{followersets}\F_k(X,a):=\{b\in B_k(X):\ ab\in B_{k+n}(X)\},\end{equation}and the {\em $k^{th}$ predecessor set} of $a$ in $X$ is\begin{equation}\label{predecessorsets}\Pp_k(X,a):=\{b\in B_k(X):\ ba\in B_{k+n}(X)\}.\end{equation}

Similarly, the $k^{th}$ follower set of $a\in B_{\text{linf}}(X)$ in $X$ is defined as \begin{equation}\label{followersets_ifinite}\F_k(X,a):=\{b\in B_k(X):\ ab\in B_{\text{linf}}(X)\}.\end{equation}

\end{defn}


\begin{defn}
Let $X\subset \Sigma_A^\Z$. We say that $X$ satisfies the {\em infinite extension property} if  $x\in X^{\text{fin}}\setminus\{\O\}$ implies that $|\F_1(X,(x_{i+l(x)})_{i\leq 0})|=\infty$.
\end{defn}

\begin{rmk} Note that for one-sided shift spaces the definition of the infinite extension property includes the empty sequence (see \cite[Definition 3.1]{OTW14}), while here there is no condition required on the empty sequence (indeed, it is not even clear what it would mean to ``follow'' the empty sequence).
\end{rmk}

\begin{lem}\label{infdense2}
Let $X\subseteq\Sigma_A^\Z$ be a set such that $\s(X)=X$. Then $X^{\text{inf}}$ is dense in $X$ if and only if $X^{\text{fin}}\setminus\{\O\}$ satisfies the infinite extension property. Furthermore, $\O$ is an accumulation point of $X^{\text{inf}}$ if and only if $|X^{\text{inf}}|=\infty$.
\end{lem}

\begin{proof}Suppose that $X^{\text{inf}}$ is dense. If $X^{\text{fin}}\setminus\{\O\}=\emptyset$, then there is nothing to do. So assume that there exists $x\in X^{\text{fin}}$, $x\neq \O$.

From Lemma \ref{nhoodbasis} the family of sets
$\{Z(x, F):\ F\text{ is a finite subset of }A\}$ is a neighborhood basis for $x$. Let $(a_i)_{i\in\N}$ be a sequence of symbols from $A$ without repetitions and, for each $n\in\N$, define $F_n:=\{a_i:\ i\leq n\}$. Since $X^{\text{inf}}$ is dense in $X$ we have that $Z(x, F_n)\cap X^{\text{inf}}\neq\emptyset$, for all $n\in\N$. Hence we can take a sequence $(y^n)_{n\in\N}$ in $X^{\text{inf}}$ where $y^n\in Z(x, F_n)\cap X^{\text{inf}}$, for all $n$. It follows that $(y^n_i)_{i\leq l(x)} = (x_i)_{i\leq l(x)}$ for all $n\in\N$, and $y^m_{l(x)+1}\neq y^n_{l(x)+1}$ for all $m\neq n$, which means that $\{y^n_{l(x)+1}\}_{n\in\N}$ is an infinite subset of $\F_1( X,x^*)$, where $x^*=(x_{i+k})_{i\leq 0}$.

Conversely, suppose that $X^{\text{fin}}\setminus\{\O\}$ satisfies the infinite extension property. Take $x\in X^{\text{fin}}\setminus\{\O\}$ and let $F\subset A$ be a finite set. Since sets of the type $Z(x, F)$ form a neighbourhood basis for $x$, we just need to prove that $Z(x, F)\cap X^{\text{inf}}$ is nonempty.
Due to the infinite extension property we have that $|\mathcal{F}_1( X, x^*)| = \infty$, which implies that we can find $y\in  X$ such that $y\neq x$ and $y\in Z(x, F)\cap X$. If $y\in  X^{\text{inf}}$ then we are done. Otherwise, we have that $l(y)>l(x)$ and $Z(y)\cap X\subsetneq  Z(x, F)\cap X$. Set $y^1:=y$ and use the infinite extension property again to pick $y^2\in Z(y^1)\cap X$ such that $y^2\neq y^1$. Again, if $y^2\in  X^{\text{inf}}$ then we are done. Otherwise, we have that $l(y^2)>l(y^1)$ and $Z(y^2)\cap X\subsetneq  Z(y^1)\cap X$ and we can proceed recursively. So, either we will find some infinite sequence in some step proving the result, or we will define a sequence $(y^i)_{i\in\N}$ in $X^{\text{inf}}$ such that $Z(y^{i+1})\cap X\subsetneq Z(y^i)\cap X\subsetneq Z(x, F)\cap X$ for all $i\in\N$. In the last case, since the cylinders are closed, it follows that there exists an infinite sequence $z\in\bigcap_{i\in\N} \big[Z(y^i)\cap X\big]\subset Z(x, F)\cap X$.

Finally, for the last statement, notice that if $|X^{\text{inf}}|<\infty$ then $\O$ is not an accumulation point of $X^{\text{inf}}$. Conversely, if $|X^{\text{inf}}|=\infty$, then the fact that $\sigma(X) = X$ precludes the possibility that $X^{\text{inf}}$ is contained in a finite union of generalized cylinders. Therefore, by Corollary \ref{accumulation_at_O}, we have that $\O$ is an accumulation point of $X^{\text{inf}}$.

\end{proof}

We now define two-sided shift spaces.

\begin{defn}\label{2-sided_shifts} A set $\Lambda\subseteq\Sigma_A^\Z$ is said to be a {\em two-sided shift space over $A$} if the following three properties hold:
\begin{enumerate}
\item[1 -] $\Lambda$ is closed with respect to the topology of $\Sigma^\Z_A$;
\item[2 -] $\Lambda$ is invariant under the shift map, that is, $\s(\Lambda) = \Lambda$;
\item[3 -] $\Lambda^{\text{inf}}$ is dense in $\Lambda$.
\end{enumerate}
\end{defn}

If there is no possibility of confusion, we will simply call such a subset a {\em shift space}. Given a shift space $\Lambda\subseteq\Sigma_A^\Z$, we endow it with the
subspace topology from $\Sigma_A^\Z$.

\begin{rmk} Ott-Tomforde-Willis one-sided shift spaces are closed subsets of $\Lambda\subset\Sigma_A^\N$ such that $\s(\Lambda)\subset\Lambda$ and such that all elements of  $\Lambda^{\text{fin}}$ satisfy the infinite extension property (see \cite[Definition 3.1]{OTW14}). In their setting, the infinite extension property is a condition equivalent to the density of $\Lambda^{\text{inf}}$ in $\Lambda$ (see \cite[Proposition 3.8]{OTW14}). 
\end{rmk}

\begin{rmk} According to the above definition of two-sided shift spaces, even if $|\LA|<\infty$, we have that $\O$ is an accumulation point of the shift space whenever $|\Lambda|=\infty$. Hence two-sided shift spaces over finite alphabets, defined according to Definition \ref{2-sided_shifts}, do not coincide with classical shift spaces over finite alphabets.
\end{rmk}

\begin{prop}\label{closure_shift}
Let $X\subset\Sigma_A^{\Z\ \text{inf}}$ be a set such that $\s(X)=X$. Then $\bar X$ is a shift space, where $\bar X$ denotes the closure of $X$ in $\Sigma_A^\Z$.
\end{prop}

\begin{proof}
It is clear that $\bar X$ is closed and shift invariant. Furthermore $X\subset\bar X^\text{inf}$ and since $X$ is dense in $\bar X$, the result follows.

\end{proof}

\begin{rmk}
Note that in the previous proposition we can have $X$ being a proper subset of $\bar X^\text{inf}$. For example, if $a,b\in A$ are two distinct letters and $X:=\{(x_i)_{i\in\Z}:\ \exists k\in\Z\text{ such that } x_i=a\ \forall i\leq k \text{ and} \ x_i=b\ \forall i> k\}$, then we have that $\bar X^\text{inf}=X\cup\{(x_i)_{i\in\Z}:\ x_i=a\ \forall i\in\Z\}$.
\end{rmk}

Two-sided shift spaces can also be alternatively defined from a set of forbidden words in an analogous way to \cite[Definition 3.11]{OTW14}. However, in the two-sided case, in addition to forbidding subblocks, we need to consider the possibility that entire pasts are forbidden. 

\begin{defn} Given $\mathbf{F}\subset A^{\Z^-}\cup\bigcup_{k\geq 1}A^k$, let
\begin{equation}\label{forbidden_words_2-sided}\begin{array}{c} X_{\mathbf{F}}^{\text{inf}} :=
\{ x \in \Sigma_A^{\Z\ \text{inf}}: \big[B_{\text{linf}}(\{x\})\cup B(\{x\})\big]\cap\mathbf{F}=\emptyset \}\\\\ \text{and}\\\\ X_{\mathbf{F}}^{\text{fin}} :=\left\{\begin{array}{lcl} 
\emptyset, &\ if& |X_{\mathbf{F}}^{\text{inf}}|<\infty,\\\\
\{ x \in \Sigma_A^{\Z\ \text{fin}}: |\F_1 (X_{\mathbf{F}}^{\text{inf}},(x_{i-l(x)})_{i\leq 0})|=\infty \}\cup\{\O\},&\ if& |X_{\mathbf{F}}^{\text{inf}}|=\infty.\end{array}\right.
\end{array}
\end{equation}
Define $X_{\mathbf{F}}:=X_{\mathbf{F}}^{\text{inf}}\cup X_{\mathbf{F}}^{\text{fin}}$.

\end{defn}

\begin{prop}\label{shiftforbiddenwords}  Given $\mathbf{F}\subset A^{\Z^-}\cup\bigcup_{k\geq 1}A^k$,  $X_{\mathbf{F}}$ is a shift space. Conversely, if $\Lambda\subseteq\Sigma_A^\Z$ is a shift
space then $\Lambda = X_{\mathbf{F}}$ for some $\mathbf{F}\subset A^{\Z^-}\cup\bigcup_{k\geq 1}A^k$.
\end{prop}

\begin{proof}
We notice that the definition of $X_{\mathbf{F}}$ implies that $\s(X_{\mathbf{F}})=X_{\mathbf{F}}$. Furthermore, by  Lemma \ref{infdense2}, $X_{\mathbf{F}}^{\text{inf}}$  is dense in $X$. Hence, to prove the first claim,  we just need to show that $(X_{\mathbf{F}})^c$ is open, which is clear when $|X_{\mathbf{F}}^{\text{inf}}|<\infty$. So, assume that $|X_{\mathbf{F}}^{\text{inf}}|=\infty$ and let $x\notin X_{\mathbf{F}}$. We now have two cases:
\begin{enumerate}
\item If $l(x) = \infty$, then either:
	\begin{enumerate}
	\item There exists $f\in \mathbf{F}\cap A^{\Z^-}$ and $n\in \Z$ such that $f_i = x_i$, for all $i\leq n$, in which case we have that $x\in Z(f)\subseteq (X_{\mathbf{F}})^c$, or;
	\item There exists $f\in \mathbf{F}\cap \bigcup_{k\in\N}A^k$ such that $f_i = x_i$, for all $m\leq i\leq n$, for some $m,n\in\Z$, in which case we have that $x\in Z(\ldots x_m\ldots x_n)\subset (X_{\mathbf{F}})^c$.
	\end{enumerate}
\item If $-\infty<l(x)<\infty$ then we again have to consider two options:
	\begin{enumerate}
	\item If $x\notin B_{\text{linf}}(X^{\text{inf}}_{\mathbf{F}})$ then we have that $x\in Z(x)\subset (X_{\mathbf{F}})^c$.
	\item If instead we have that $x\in B_{\text{linf}}(X^{\text{inf}}_{\mathbf{F}})$, but $|\mathcal{F}_1(X^{\text{inf}}_{\mathbf{F}}, x)|< \infty$, then, setting $F:= \mathcal{F}_1(X^{\text{inf}}_{\mathbf{F}}, x)$, we obtain that $x\in Z(x,F)\subset (X_{\mathbf{F}})^c$.
	\end{enumerate}
\end{enumerate}
This shows that $(X_{\mathbf{F}})^c$ is open and so $X_{\mathbf{F}}$ is closed.

To prove the second part, given a shift space $\Lambda$ we must find a set $\mathbf{F}$ such that $\Lambda = X_{\mathbf{F}}$. To this end, take
$$
\mathbf{F} = \left[ A^{\Z^-}\cup\bigcup_{k\geq 1}A^k\right]\setminus \left[B_{\text{linf}}(\Lambda)\cup B(\Lambda)\right]
$$
We first show that $X_{\mathbf{F}}^{\text{inf}} = \Lambda^{\text{inf}}$. The inclusion $\Lambda^{\text{inf}}\subset X_{\mathbf{F}}^{\text{inf}}$ is trivial. To prove the other inclusion, take $y\in X_{\mathbf{F}}^{\text{inf}}$ and, by contradiction, suppose that $y\in \Lambda^c$. Since $\Lambda$ is closed, we can find an open set containing $y$ disjoint from $\Lambda$. From Lemma \ref{nhoodbasis}, we can find a prefix $z$ of $y$ such that $Z(z)$ is disjoint from $\Lambda$. This implies that $z$ cannot be in $B_{\text{linf}}(\Lambda)$ and so $z\in \mathbf{F}$, which contradicts $y\in X_{\mathbf{F}}^{\text{inf}}$. Hence $X_{\mathbf{F}}^{\text{inf}} = \Lambda^{\text{inf}}$.
Now, by Lemma \ref{infdense2}, we have that $\Lambda=\overline{\Lambda^{\text{inf}}}=\overline{X_{\mathbf{F}}^{\text{inf}}}=X_{\mathbf{F}}$.

\end{proof}

The last proposition directly implies the following corollary.

\begin{cor}\label{language+linf} Let $\Lambda,\Gamma\subseteq\Sigma_A^\Z$ be shift spaces. Then $\Lambda=\Gamma$ if and only if $\left[B_{\text{linf}}(\Lambda)\cup B(\Lambda)\right]=\left[B_{\text{linf}}(\Gamma)\cup B(\Gamma)\right]$.
\end{cor}

\qed

Note that for standard shift spaces over finite alphabets, as well as for Ott-Tomforde-Willis one-sided shift spaces, two shift spaces are equal if and only if they have the same language. Corollary \ref{language+linf} above gives an analogue of this result for our two-sided shift spaces, but we need to take in account not only the language, but also restrictions on infinitely many coordinates to the left.

Unlike the one-sided case, in our setting the shift map is continuous everywhere. In fact $\s:\Sigma_A^\Z\to \Sigma_A^\Z$ is a homeomorphism. Hence, in contrast to the one-sided case,  $(\Sigma_A^\Z, \s)$ may be analyzed as a topological dynamical system.

\begin{defn}\label{shift_conjugate}
Two shift spaces $\Lambda\subseteq\Sigma^\Z_A$ and $\Gamma\subseteq\Sigma^\Z_B$ are said to be topologically conjugate if the dynamical systems $(\Lambda,\s)$ and $(\Gamma,\s)$ are topologically conjugate, that is, if there exists a continuous invertible map $\varphi:\Lambda\to\Gamma$ such that $\varphi\circ\s=\s\circ\varphi$.
\end{defn}

Definition \ref{shift_conjugate} is the standard definition of equivalence in dynamical systems and should be contrasted with \cite[Definition 4.1]{OTW14}, where two one-sided shift spaces are said to be conjugate if there is a bijective continuous shift-commuting map between them which preserves length. It should be noted that we do not require that a conjugacy preserves length.

We will say that a given nonempty shift $\Lambda$ is a:
\begin{itemize}

\item \uline{{\sc shift of finite type}}: if $\Lambda=X_\mathbf{F}$ for some finite $\mathbf{F}$;

\item \uline{{\sc finite-step shift}}: if $\Lambda=X_\mathbf{F}$ for some $\mathbf{F}\subset \bigcup_{k\geq 1}A^k$. In the case that $\Lambda=X_\mathbf{F}$ for some $\mathbf{F}\subseteq A^{M+1}$, we will say that $\Lambda$ is a $M$-step shift, and in particular, if $M=1$ we also say that $\Lambda$ is a {\em Markov} (or {\em Markovian}) shift;

\item \uline{{\sc infinite-step shift}}: if it is not a finite-step shift;

\item \uline{{\sc edge shift}}: Let $G=(E,V,i,t)$ be a directed graph with no sources and sinks and define $\Lambda(G)\subseteq \Sigma_E^\Z$ as the closure of the set of all possible bi-infinite walks on $G$, that is,
\begin{equation}\label{edgeshift}\Lambda(G):=\overline{\left\{(e_i)_{i\in\Z}:\ t(e_i)=i(e_{i+1})\ \forall i\in\Z\right\}}.
\end{equation}
By Proposition \ref{closure_shift}, it follows that $\Lambda(G)$ is a shift space. In particular, one can check that $\Lambda(G)^{\text{inf}}$ coincides with the set of all possible bi-infinite walks on $G$, while $\Lambda(G)^{\text{fin}}$ can be identified with with all left infinite walks on $G$ which end on a vertex of infinite order (that is, a vertex which emits infinitely many edges).
  Furthermore, $\Lambda(G)=X_\mathbf{F}$, where $\mathbf{F}:=\{ef:\ t(e)\neq i(f)\}$;

\item \uline{{\sc row-finite shift}}: if for all $a\in\LA$ we have that $\F_1(\Lambda,a)$ is a finite set;

\item \uline{{\sc column-finite shift}}: if for all $a\in\LA$ we have that $\Pp_1(\Lambda,a)$ is a finite set.
\end{itemize}

For standard shift spaces (with finite alphabet) the class of finite-step shifts only contains $M$-step shifts, and the classes of shifts of finite type, finite-step shifts and edge shifts all coincide (of course, every subshift over a finite alphabet is both row-finite and column finite). However, this is not true when the alphabet is infinite \cite[Remark 5.19]{OTW14}. Furthermore, standard shift spaces and Ott-Tomforde-Willis shift spaces cannot be infinite-step shifts \cite[Theorem 3.16]{OTW14}.

\begin{rmk} If $\Lambda\subseteq \Sigma_A^\Z$ is a row-finite shift space then $\Lambda^{\text{fin}}\subset \{\O\}$ but, unlike the one sided shift setting, the converse may not be true. For example the shift space $X_\mathbf{F}$ defined in Examples \ref{minimal_examples}.b and \ref{minimal_examples}.c are not row finite but the only finite sequence contained in them is $\O$.
\end{rmk}

\section{Higher block shifts}\label{SEC:SBC}

\subsection{Sliding block codes}

We now turn to the natural problem of identifying the continuous shift-commuting maps between our shift spaces. In \cite{GoncalvesSobottkaStarling2015} we defined and classified such maps between one-sided shifts over infinite alphabets, and we follow a similar construction here. For this we shall use the notion of {\em finitely defined sets}.

\begin{defn}\label{pseudo_cylinders} Let $X\subset \Sigma_A^\Z$, let $k\leq\ell\in\Z$ be two integers, and let $b=(b_1\ldots b_{-k+\ell+1})\in B(\Sigma_A^\Z)$. A {\em pseudo cylinder} of $X$ is a set
$$[b]_{k}^\ell:=\{(x_i)_{i\in\Z}\in X: (x_{k}\ldots x_\ell)=b\}.$$

We say that the pseudo cylinder $[b]_{k}^\ell$ has {\em memory} $K$ and {\em anticipation} $L$, if $-K\leq \min \{0,k\}$ and $L\geq \max\{0,\ell\}$. The {\em least memory} of $[b]_{k}^\ell$ is $-\min \{0,k\}$, while the {\em least anticipation} of $[b]_{k}^\ell$ is
$\max\{0,\ell\}$.

We will adopt the convention that the empty set is a pseudo cylinder of $X$ whose memory and anticipation are zero.

\end{defn}

We remark that the definition of pseudo cylinder of $X$ corresponds to the definition of cylinders around the position 0 for standard shift maps over finite alphabets (for one-sided shifts, $k=0$). On the other hand, for a one-sided shift space $\Lambda$ cylinders of the form $Z(x)\cap\Lambda$, with $x\neq\O$, satisfy Definition \ref{pseudo_cylinders}.

\begin{prop}\label{finite_intersec-pcylinders} Let $[b]_{k}^{\ell}$ and $[c]_{m}^{n}$ be two pseudo cylinders in $X\subset\Sigma_A^\Z$. Then $[b]_{k}^{\ell}\cap[c]_{m}^{n}$ is a union of pseudo cylinders in $X$.
\end{prop}

\begin{proof}
Suppose $[b]_{k}^{\ell}\cap[c]_{m}^{n}\neq\emptyset$, since if not, then the result is proved. Without loss of generality we can assume $\ell\leq n$ and therefore we have three cases:

\begin{description}

\item[$m\leq k\leq \ell\leq n$:] In this case $[c]_{m}^{n}\subset [b]_{k}^{\ell}$ and then $[b]_{k}^{\ell}\cap[c]_{m}^{n}=[c]_{m}^{n}$.

\item[$k< m\leq \ell\leq n$:] In this case $[b]_{k}^{\ell}\cap[c]_{m}^{n}=[d]_{k}^{n}$, where $d=b$ if $\ell=n$, while in the case $\ell<n$ the word $d$ is obtained by concatenating the word $b=(b_1\ldots b_{\ell-k+1})$ with $(c_{\ell-m+2}\ldots c_{n-m+1})$.

\item[$k\leq \ell< m\leq n$:] In this case, if $m-\ell=1$, then $[b]_{k}^{\ell}\cap[c]_{m}^{n}=[d]_{k}^{n}$, where $d$ is the concatenation of $b$ and $c$. On the other hand, if  $m-\ell>1$, then taking $\Omega:=\{f\in B_{m-\ell-1}(\Sigma_A^\Z):\ bfc \in B(\Sigma_A^\Z)\}$, it follows that $$[b]_{k}^{\ell}\cap[c]_{m}^{n}=\bigcup_{f\in\Omega}[bfc]_{k}^{n}.$$

\end{description}

\end{proof}

\begin{prop}\label{intersec-pcylinders} Let $\{[b^i]_{k_i}^{\ell_i}\}_{i\in\N}$ be a family of pseudo cylinders in $X\subset\Sigma_A^\Z$ such that, for all $b^i=(b^i_1\ldots b^i_{-k_i+\ell_i+1})$ we have that $b^i_{-k_i+\ell_i+1}\neq \o$. Then
\begin{enumerate}

\item[(i)] $\bigcup_{i\in \N} [b^i]_{k_i}^{\ell_i}$ is an open subset of $X$;

\item[(ii)] If there exists $L\in\N$ such that $\ell_i\leq L$ for all $i\in\N$, then $\bigcap_{i\in N} [b^i]_{k_i}^{\ell_i}$ is an open subset of $X$.

\end{enumerate}
\end{prop}

\begin{proof}

Item $(i)$ is direct, since each pseudo cylinder $[b^i]_{k_i}^{\ell_i}$ is an open subset of $X$.

To prove $(ii)$, notice that if $\big(\bigcap_{i\in N} [b^i]_{k_i}^{\ell_i}\big)\cap X=\emptyset$ we are done. So, suppose $\bigcap_{i\in N} [b^i]_{k_i}^{\ell_i}\neq\emptyset$. Then, for all $x\in\bigcap_{i\in N} [b^i]_{k_i}^{\ell_i}$, it follows that $$Z\big((x_n)_{n\leq \max\{\ell_i\}}\big)\cap X\subset\bigcap_{i\in N} [b^i]_{k_i}^{\ell_i}.$$

\end{proof}

We now define the subsets of a shift space from which we will construct our sliding block codes.

\begin{defn}\label{finitely_defined_set}
Given $C\subset X\subset \Sigma_A^\Z$, we will say that $C$ is {\em finitely defined in $X$} if there exist two collections of pseudo cylinders of $X$, $\{[b^i]_{k_i}^{\ell_i}\}_{i\in I}$ and $\{[d^j]_{m_j}^{n_j}\}_{j\in J}$,  such that

\begin{equation}\label{eq:finitely_defined_set}\begin{array}{l}
C=\bigcup_{i\in I} [b^i]_{k_i}^{\ell_i},\\\\
C^c= \bigcup_{j\in J}[d^j]_{m_j}^{n_j}.\end{array}
\end{equation}
\end{defn}

We recall that Definition \ref{finitely_defined_set} above is equivalent to  Definition 3.1 in \cite{GoncalvesSobottkaStarling2015} by taking $k=0$. In other words, $C$ is a finitely defined set of $X$ if and only if for each $x\in X$ there exist $k,\ell\in\N$ such that the knowledge of $(x_{-k}\ldots x_\ell)$ allows us to know whether $x$ belongs or not to $C$. In particular, $X$ and $\emptyset$ are finitely defined sets of $X$. We note that, in contrast to the one-sided case of \cite{GoncalvesSobottkaStarling2015}, the set $\{\O\}$ is {\em not} finitely defined, since knowing that a finite number of entries of a sequence are all $\o$ is not enough to guarantee that it is equal to $\O$.

Note that if $C$ is a finitely defined set of $X$, then there exist infinitely many ways to write $C$ and $C^c$ as union of pseudo cylinders of $X$.

\begin{defn}\label{memory_and_anticipation}
Let $C$ be a finitely defined set in $X$. Then we say that $C$ has {\em memory} $k$ and {\em anticipation} $\ell$ if there exists a collection  $\{[b^i]_{k_i}^{\ell_i}\}_{i\in I}$, which satisfies \eqref{eq:finitely_defined_set}, with

\begin{equation}\label{eq:memory_and_anticipation}\begin{array}{l}-k\leq \inf \{0,k_i:\ i\in I\} \\\\\ell\geq \sup\{0,\ell_i:\ i\in I\}.\end{array}\end{equation}

We say that $C$ has {\em least memory} $k$ and {\em least anticipation} $\ell$ in the case that $k$ and $\ell$ are the least integers for which \eqref{eq:memory_and_anticipation} holds for some collection of pseudo cylinders satisfying \eqref{eq:finitely_defined_set}.

If for any collection $\{[b^i]_{k_i}^{\ell_i}\}_{i\in I}$ of pseudo cylinders satisfying \eqref{eq:finitely_defined_set} we have that $\inf_{i\in I} k_i=-\infty$ or $\sup_{i\in I} \ell_i=\infty$ then we say that $C$ has {\em unbounded memory} or {\em unbounded anticipation}, respectively.
\end{defn}

We now show that the class of finitely defined sets is closed under taking finite unions and intersections.

\begin{prop}\label{fdunion}
A finite union or intersection of finitely defined sets is also a finitely defined set.
\end{prop}

\begin{proof}
Let $\{C_p\}_{p=1,\ldots,n}\subset X\subset\Sigma_A^\Z$ be a collection of finitely defined subsets of $X$.

Let us prove that  $K:=\bigcup_{p=1}^n C_p$ is a finitely defined set in $X$. For each $p=1,\ldots,n$ let $\{[b^{p,i}]_{k_i^p}^{\ell_i^p}\}_{i\in I_p}$ and $\{[d^{p,j}]_{m_j^p}^{n_j^p}\}_{i\in J_p}$ be two collections of pseudo cylinders for which $C_p=\bigcup_{i\in I_p}[b^{p,i}]_{k_i^p}^{\ell_i^p}$ and $C_p^c=\bigcup_{j\in J_p}[d^{p,j}]_{m_j^p}^{n_j^p}$. Clearly, $K$ is a union of pseudo cylinders. On the other hand,

$$K^c= \left(\bigcup_{p=1}^n C_p\right)^c=\bigcap_{p=1}^n C_p^c= \bigcap_{p=1}^n
  \bigcup_{j\in J_p}[d^{p,j}]_{m_j^p}^{n_j^p}= \bigcup_{j\in J_p} \bigcap_{p=1}^n [d^{p,j}]_{m_j^p}^{n_j^p}
$$
and since Proposition \ref{finite_intersec-pcylinders} assures that a finite intersection of pseudo cylinders can be written as an union of pseudo cylinders, we have that $K^c$ is a union of pseudo cylinders, proving that $K$ is a finitely defined set.

To prove that $\bigcap_{p=1}^n C_p$ is finitely defined, we just need to observe that $\left(\bigcap_{p=1}^n C_p\right)^c=\bigcup_{p=1}^n C_p^c$ and, since each $C_p^c$ is a finitely defined set, the result follows from the first part of this proof.
\end{proof}

\begin{prop} A non-empty shift space $\Lambda\subsetneq\Sigma_A^\Z$ is never a finitely defined set of $\Sigma_A^\Z$.
\end{prop}

\begin{proof}
Suppose $\Lambda=X_\mathbf{F}$, for some $\mathbf{F}\subset A^{\Z^-}\cup\bigcup_{k\geq 1}A^k$. To check whether some point $x\in\Sigma_A^\Z$ belongs or not to $\Lambda$ we need to check if any subblock of $x$ belongs to $\mathbf{F}$, which we cannot do with only knowledge of $(x_{-k}\ldots x_\ell)$ for some $k,\ell\in\N$.

\end{proof}

Using Proposition \ref{fdunion} we can now define a natural class of shift-commuting maps between two-sided shift spaces.
\begin{defn}\label{defn_sliding block code}
Let $A$ and $B$ be alphabets and $\Lambda \subset \Sigma_A^\Z$ and $\Gamma \subset \Sigma_B^\Z$ be shift spaces. Suppose that $\{C_a\}_{a\in \LAG\cup\{\o\}}$ is a pairwise
disjoint partition of $\Lambda$, such that:
\begin{description}\addtolength{\itemsep}{-0.5\baselineskip}

\item[\em 1.] for each $a\in \LAG\cup\{\o\}$ the set $C_a$ is finitely defined in $\Lambda$, and

\item[\em 2.] $C_{\o}$ is shift invariant (that is, $\s(C_{\o})\subset C_{\o}$).

\end{description}

We will say that a map $\Phi:\Lambda\to\Gamma$ is a {\em sliding block code} if

\begin{equation}\label{LR_block_code}\bigl(\Phi(x)\bigr)_n=\sum_{a\in \LAG\cup\{\o\}}a\mathbf{1}_{C_a}\circ\sigma^{n}(x),\quad \forall x\in\Lambda,\ \forall n\in\Z, \end{equation} where $\mathbf{1}_{C_a}$ is the
characteristic function of the set $C_a$ and $\sum$ stands for the symbolic sum.

For a sliding block code $\Phi$ let $k_a$ and $\ell_a$ be the memory and the anticipation of each $C_a$, respectively and let $k:=\sup_{a\in L_\gamma\cup\{\o\}}k_a$ and  $\ell:=\sup_{a\in
L_\gamma\cup\{\o\}}\ell_a$. If $k,\ell<\infty$ we will say that $\Phi$ is a $k+\ell+1$-block code with memory $k$ and anticipation $\ell$, while if $k=\infty$ or $\ell=\infty$, we will say that $\Phi$ has unbounded memory or anticipation, respectively.
\end{defn}

Intuitively speaking, $\Phi:\Lambda\to \Gamma$ is a sliding block code if for all $n\in\Z$ and $x\in\Lambda$, there exist $k,\ell\geq 0$ such that we only need to know $(x_{n-k}x_{n-k+1} \ldots x_{n+\ell})$ to determine $\bigl(\Phi(x)\bigr)_n$. We remark that $k$ and $\ell$ above do not depend on the value of $n$, but do depend on the configuration of $x$ around $x_n$.

We will also say that sliding block codes are maps given by a {\em local rule} (which is implicitly given by \eqref{LR_block_code} and, in some cases, can be given explicitly). Also, note that the sets $C_a$ in the definition of sliding block codes can be written as $C_a=\Phi^{-1}\big([a]_0^0\big)$, where for each $a\in \LAG\cup\{\o\}$, $[a]_0^0$ is the pseudo cylinder of $\Gamma$ consisting of all elements that have the letter $a$ at coordinate zero.

We now give some examples and nonexamples of sliding block codes.

\begin{ex}\label{sliding block code_examp}

a) If $\Lambda$ is any shift space then the shift map $\s:\Lambda\to\Lambda$ is a continuous and invertible sliding block code with memory 0 and anticipation 1.
Its local rule is given by $$\bigl(\s(x)\bigr)_n=\sum_{a\in \LA\cup\{\o\}}a\mathbf{1}_{C_a}\circ\sigma^{n}(x),$$
where $C_a=\{x\in\Lambda:\ x_1=a\}=\bigcup_{b\in\LA}[ba]_0^1$ for all $a\in\LA\cup\{\o\}$.

b) Let
$A=\N$ and let $\Phi:\Sigma_A^\Z\to\Sigma_A^\Z$
be given by
$$\big(\Phi(x)\big)_n:=\sup_{j\leq n} x_j,\hspace{1cm}\text{for all }x\in \Sigma_A^\Z,  n\in\Z,$$
with the convention that $\o> a$ for all $a\in A$.
It is immediate that $\Phi$ is not a sliding block code, but $\Phi$ is continuous and shift commuting.

c) Let $A=\Z$ and consider the shift space $\Lambda:=\{(x_i)_{i\in\Z}\in\Sigma_A^\Z:\ x_{i+1}\geq x_i\ \forall i\in\Z\}$, with the convention that $\o>a$ for all $a\in A$. Define $\Phi:\Lambda\to\Lambda$ by

$$\big(\Phi(x)\big)_n:=\begin{cases} \o &\text{if } x_n=\o,\\ x_{x_n}& \text{otherwise.}\end{cases}$$
Then $\Phi$ is a sliding block code with unbounded memory and anticipation.

d)  Let $A=\N$ and define $\Phi:\Sigma_A^\Z\to \Sigma_A^\Z$ by
$$\big(\Phi(x)\big)_n:=\begin{cases} (x_n+1)/2 &\text{if } x_n \text{ is odd},\\
                                           x_n /2 &\text{if } x_n \text{ is even}\\
                                           \o      &\text{if } x_n =\o.\end{cases}$$
Then $\Phi$ is a sliding block code which is onto but not one-to-one. Furthermore, $\Phi$ is not continuous. Indeed, consider the sequence $(x^i)_{i\geq 1}$ where  $x_{-i}^i=2$ and $x_n^i=1$ if $n\neq -i$. Then $x^i$ converges to $\O$, but $\Phi(x^i)=(\ldots 111\ldots)$ for all $i\geq 1$.

\end{ex}

In \cite[Theorems 3.16 and 3.17]{GoncalvesSobottkaStarling2015} we gave necessary and sufficient conditions under which one can obtain an analogue of the Curtis-Hedlund-Lyndon Theorem for Ott-Tomforde-Willis one-sided shift spaces. However, due to the features of the topology assumed here for two-sided shift spaces of $\Sigma_A^\Z$, we cannot use the techniques of  \cite{GoncalvesSobottkaStarling2015} to obtain a similar result. For instance, in the one-sided case the shift map is only continuous at $\O$ for column-finite shifts (which is explained by  \cite[Theorem 3.16]{GoncalvesSobottkaStarling2015}), while in the two-sided case the shift map is always continuous at $\O$. Example \ref{sliding block code_examp}.d is another example which illustrates the fact that the Curtis-Hedlund-Lyndon Theorems obtained in \cite{GoncalvesSobottkaStarling2015} cannot be applied for two-sided shift spaces. In fact, while for the one-sided case a sliding block code $\Phi:\Sigma_A^\N\to\Sigma_A^\N$ defined with the same rule as Example \ref{sliding block code_examp}.d  is continuous everywhere, for the two-sided case it is not continuous at $\O$.

One can easily prove that sliding block codes are always shift commuting. In fact, the following results and their proofs are analogous to Proposition 3.12 and Corollaries 3.13--3.15 in \cite{GoncalvesSobottkaStarling2015}.

\begin{prop}\label{sliding block code->shift_commuting} Any sliding block code commutes with the shift map.
\end{prop}

\begin{cor}\label{sliding block code->preserves_period}
If $\Phi:\Lambda\to\Gamma$ is a sliding block code and $x\in\Lambda$ is a sequence with period $p\geq 1$ (that is, $\s^p(x)=x$), then $\Phi(x)$ has also period $p$.
\end{cor}

\begin{cor}\label{sliding block code->emptysequence goes to constant}
If $\Phi:\Lambda\to\Gamma$ is a sliding block code, then $\Phi(\O)$ is a constant sequence (that is, $\Phi(\O)=(\ldots ddd\ldots)$ for some $d\in \LAG\cup\{\o\}$). Furthermore, if $\Phi(\O)=\O$, then the image of a finite sequence $x\in\Lambda^{\text{fin}}$ by $\Phi$ is a finite sequence of $\Gamma$.
\end{cor}

As alluded to already, in the finite alphabet case a map between shift spaces is shift commuting and continuous if and only if it is a sliding block codes -- this is the Curtis-Lyndon-Hedlund Theorem. The following theorems give sufficient conditions for our sliding block codes to be continuous.

\begin{theo}\label{continuoussliding block code}
Let $\Phi:\Lambda\to\Gamma$ be a sliding block code such that $\Phi(\O)=\O$, and suppose there exists $L\in\Z$ such that $\Phi=\sum_{a\in \LAG\cup\{\o\}}a\mathbf{1}_{C_a}\circ\sigma^{n}(x)$ and such that for each $a\in \LAG$, $C_a=[c^a]_{k_a}^{L}$ for some $c^a_{-k_a+L+1}\neq \o$. Then $\Phi$ is continuous. Additionally, if $C_{\o}=\displaystyle\bigcup_{\tiny\begin{array}{c}b\in B(\Lambda):\\b_{-k_b+L+1}= \o\end{array}}[b]_{k_b}^{L}$, then $\Phi$ is a homeomorphism onto its image.
\end{theo}

\begin{proof}

Given $y=(y_n)_{n\in\Z}\in\Gamma$, for all $K\leq l(y)$ we have that

\begin{equation}\label{EQcontinuoussliding}
\begin{array}{lll}\Phi^{-1}\Big(Z\big((y_n)_{n\leq K}\big)\cap\Gamma\Big)&=\displaystyle\Phi^{-1}\left(\bigcap_{n\leq K}[y_n]_n^n\right)
&=\displaystyle\bigcap_{n\leq K}\Phi^{-1}\big([y_n]_n^n\big)\\
&= \displaystyle\bigcap_{n\leq K}\Phi^{-1}\circ\s^{-n}([y_n]_0^0\big)
&=\displaystyle\bigcap_{n\leq K}\s^{-n}\circ\Phi^{-1}([y_n]_0^0\big)\\
&=\displaystyle\bigcap_{n\leq K}\s^{-n}\big(C_{y_n}\big)
&=\displaystyle\bigcap_{n\leq K}\s^{-n}\big([c^{y_n}]_{k_{y_n}}^{L}\big)\\
&=\displaystyle\bigcap_{n\leq K}[c^{y_n}]_{k_{y_n}+n}^{L+n}.&
 \end{array}
\end{equation}

Thus $\Phi^{-1}\Big(Z\big((y_n)_{n\leq K}\big)\cap\Gamma\Big)$ is a clopen set since it is either the empty set or a generalized cylinder defined on the coordinates less than or equal to $K+L$.
Furthermore, if $y^j\in\Gamma$ and $K_j\leq l(y^j)$ for $j=1\ldots,N$, then

$$\Phi^{-1}\left(\left[\bigcup_{j=1}^N Z((y_n^j)_{n\leq K_j})\cap\Gamma\right]^c\right)=\displaystyle\Phi^{-1}\left(\bigcap_{j=1}^N \left[Z((y_n^j)_{n\leq K_j})\cap\Gamma\right]^c\right)
=\displaystyle\bigcap_{j=1}^N \left[\Phi^{-1}\left(Z((y_n^j)_{n\leq K_j})\cap\Gamma\right)\right]^c$$
and, since each $\Phi^{-1}\left(Z((y_n^j)_{n\leq K_j})\cap\Gamma\right)$ is a clopen set, we have that $\Phi^{-1}\left(\left[\bigcup_{j=1}^N Z((y_n^j)_{n\leq K_j})\cap\Gamma\right]^c\right)$ is also a clopen set.

To finish we observe that for any finite set $F\subset A$, given a cylinder of the form $Z\big((y_n)_{n\leq K},F\big)\cap\Gamma$, we have that
$$Z\big((y_n)_{n\leq K},F\big)\cap\Gamma=Z((y_n)_{n\leq K})\cap\left[\bigcup_{f\in F} Z(\ldots,y_{K-1},y_K,f)\right]^c,$$
which implies that its inverse image by $\Phi$ is also a clopen set. Hence we conclude that $\Phi$ is continuous.\\

Now suppose additionally that $C_{\o}=\displaystyle\bigcup_{\tiny\begin{array}{c}b\in B(\Lambda):\\b_{-k_b+L+1}= \o\end{array}}[b]_{k_b}^{L}$. Since for all $y\in\Phi(\Lambda)$ we have that $\displaystyle\{y\}=\bigcap_{n\in\Z}[y_n]_n^n$, then, by a computation similar to \eqref{EQcontinuoussliding}, $\Phi^{-1}(y)$ will be a singleton. Thus, $\Phi$ is invertible on its image and, since $\Lambda$ and $\Phi(\Lambda)$ are compact, $\Phi$ is a homeomorphism.

\end{proof}

\subsection{Higher block presentation of a shift space}

Another standard construction in symbolic dynamics is the higher block presentation, where a subshift is recoded to another shift over the alphabet of all possible blocks of a given length. In this section we adapt this construction to our shift spaces.

\begin{defn}\label{defn_HBC} Given $\Sigma_A^\Z$ and $M\in\N$, denote $A^{[M]}:=B_M(\Sigma_A^\Z)$. The {\em $M^{th}$ higher block code}
is the map $\Xi^{[M]}:\Sigma_A^\Z\to\Sigma_{A^{[M]}}^\Z$ given, for all $x\in\Sigma_A^\Z$ and $i\in\Z$, by
$$\big(\Xi^{[M]}(x)\big)_i: = \begin{cases} \o &\text{if }x_i=\o\\
 \left[x_{i-M+1}\ldots x_i\right] &\text{if }x_i\neq\o.\end{cases}$$

\end{defn}

It is immediate that a higher block code is an invertible $M$-block code with anticipation 0. We now show that the $M$th higher block code is continuous, injective, and sends shift spaces to shift spaces.

\begin{prop}\label{M_Higher_Block_Code}
Let $\Lambda\subset \Sigma_A^\Z$ be a shift space and let $\Xi^{[M]}:\Sigma_A^\Z\to\Sigma_{A^{[M]}}^\Z$ be the $M^{th}$ higher block code. Then,
\begin{enumerate}

\item[(i)] $\left.\Xi^{[M]}\right|_\Lambda:\Lambda\to\Xi^{[M]}(\Lambda)$ is continuous.

\item[(ii)]  $\left.\Xi^{[M]}\right|_\Lambda:\Lambda\to\Xi^{[M]}(\Lambda)$ is invertible and its inverse is a 1-block code given by $$(y_i)_{i\in\Z}\mapsto (\tilde x_i)_{i\in\Z},$$ where $\tilde x_i=x_i$ if $y_i=[x_{i-M+1}
\ldots x_i]$ and $\tilde x_i=\o$ if $y_i=\o$;

\item[(iii)]  $\Xi^{[M]}(\Lambda)$ is a shift space in $\Sigma_{A^{[M]}}^\Z$;

\end{enumerate}
\end{prop}

\begin{proof} The statement $(i)$ follows from the fact that $\Phi$ is a sliding block code which satisfies the properties of Theorem \ref{continuoussliding block code}, while $(ii)$ is direct from the definition of $\Xi^{[M]}$.

To prove $(iii)$ we notice that, since $\Phi$ is a continuous sliding block code, we have that $\Xi^{[M]}(\Lambda)$ is closed in $\Sigma_{A^{[M]}}^\Z$ and invariant under $\s$. Thus we only need to show that $\Xi^{[M]}(\Lambda)$ satisfies the infinite extension property.

   Observe that if $y=\big([b_{i-M+1}\ldots b_i]\big)_{i\leq L}\in\Xi^{[M]}(\Lambda)^{\text{fin}}$, then it is necessarily the image of some $b=(b_i)_{i\leq L}\in\Lambda^{\text{fin}}$. Since
   $\Lambda$ has the infinite extension property, there exists a sequence $x_n = (x_i^n)_{i\in\Z}$ in $\Lambda$ such that $x^{n}_i=b_i$, for all $n\geq 1$ and $i\leq L$, and $x^n_{L+1}\neq x^m_{L+1}$ if $n\neq m$ (that
   is, $x_n$ is a sequence which converges to $b$).

   Define $y_n=(y^n_i)_{i\in\Z}:=\Xi^{[M]}(x^{n})$. It is clear that $y_n\in\Lambda^{[M]}$. Note that for all $n\geq 1$ and $i\leq L$ we have that $y^n_i=[b_{i-M+1}\ldots b_i]=y_i$, and $y^n_{L+1}=[b_{L-M+2}\ldots
   x^n_{L+1}]$. Hence, since $x^n_{L+1}\neq x^m_{L+1}$ if $n\neq m$, it follows that $y^n_{L+1}=[b_{L-M+2}\ldots x^n_{L+1}]\neq[b_{L-M+2}\ldots x^m_{L+1}]= y^m_{L+1}$ if $m\neq n$, which means that $y_n$ converges to
   $y$. Since this holds for any sequence of $\Xi^{[M]}(\Lambda)^{\text{fin}}$, $\Xi^{[M]}(\Lambda)$ satisfies the infinite extension property.

\end{proof}

Given a shift space $\Lambda\subseteq\Sigma_A^\Z$ we will call $\Lambda^{[M]}:=\Xi^{[M]}(\Lambda)$ its {\em $M^{th}$ higher block shift}. Proposition \ref{M_Higher_Block_Code} above means that any shift space $\Lambda$ is always topologically conjugate to its $M^{th}$ higher block shift, for any $M\geq 2$.\\

 Given a left infinite block $w=(w_i)_{i\leq 0}$, define $\Xi^{[M]}(w):=\big([w_{i-M+1}\ldots w_i]\big)_{i\leq 0}$, and given a finite block $w=(w_i)_{1\leq i\leq k}$ with $k\geq M$ define $\Xi^{[M]}(w):=\big([w_i\ldots w_{i+M-1}]\big)_{1\leq i \leq k-M+1}$. The next result shows that a higher block presentation of a shift given by a set of forbidden words $\mathbf{F}$ can be given by a set of forbidden words derived from $\mathbf{F}$, and is proved using the same argument seen in \cite[Proposition 1.4.3]{LindMarcus}.

 \begin{cor}\label{COR:M_Higher_Block_Code} Given a shift space $\Lambda\subset\Sigma_A^\Z$, let $\mathbf{F}\subset A^{\Z^-}\cup\bigcup_{k\geq M}A^k$ be any set such that $\Lambda=X_{\mathbf{F}}$. Then, $\Lambda^{[M]}=X_{\mathbf{F}^{[M]}}$, where $\mathbf{F}^{[M]}\subset (A^M)^{\Z^-}\cup\bigcup_{k\geq 1}(A^M)^k$ is given by
$$ \mathbf{F}^{[M]}:=\Big\{([u_i\ldots u_M],[v_1\ldots v_M])\in (A^M)^2: \ u_i\neq v_{i-1},\ 2\leq i\leq M\Big\}\cup\Big\{\Xi^{[M]}(w):\ w\in\mathbf{F}\Big\}.$$

\end{cor}

We now show that, in contrast to the finite alphabet case, taking a higher block presentation never results in a shift of finite type.
\begin{prop} Let $\Lambda\subset \Sigma_A^\Z$ be a shift space with $|\LA|=\infty$. Then for any $M\geq 2$,  the $M^{th}$ higher block shift of $\Lambda$ is never a shift of finite type.
\end{prop}

\begin{proof}
This follows from the fact that $\Lambda^{[M]}=X_{\mathbf{F}}$ for some $\mathbf{F}$ which must contains all blocks of length 2 of the form $$\Big((a_1\ldots a_M)(b_1\ldots b_M)\Big),$$ where $a_{i+1}\neq b_i$ for all $i=1,\ldots,M-1$.

\end{proof}

The following results are consequences of Proposition \ref{M_Higher_Block_Code} and their proofs are analogous to the proofs in the finite alphabet case (see \cite{LindMarcus}). We sketch only the proof for Proposition \ref{edge_shift_conjugacy}.

\begin{prop}\label{higher_block_conjugacy}
Let $\Lambda\subset \Sigma_A^\Z$ be a shift space and let $\Lambda^{[M]}$ be its $M^{th}$ higher block shift.

\begin{enumerate}

\item[(i)] If $\Lambda$ is a column-finite (or row-finite) shift, then $\Lambda^{[M]}$ is also a column-finite (or row-finite) shift;

\item[(ii)] If $\Lambda$ is a $K$-step shift, then $\Lambda^{[M]}$ is a $L$-step shift where $L=\max\{1,K-M+1\}$.

\end{enumerate}
\end{prop}

\begin{prop}\label{edge_shift_conjugacy} Any $M$-step shift of $\Sigma_A^\Z$ is topologically conjugate to some edge shift.
\end{prop}

\begin{proof}
Let $\Gamma$ be an $M$-step shift in $\Sigma_A^\Z$. If $M = 0$ then $\Gamma$ is $\Sigma_B^\Z$ for some $B\subset A$, which is the edge shift over the graph with one vertex and with edge set equal to $B$. So suppose $M\geq 0$, and as in \cite[Theorem 2.3.2]{LindMarcus}, let $E = B_M(\Gamma)$ and $V = B_{M+1}(\Gamma)$. Furthermore, define maps $t$ and $i$ from $V$ to $E$ by
$$
i(v_1v_2\cdots v_{M+1}) = v_1v_2\cdots v_{M}, \hspace{1cm}  t(v_1v_2\cdots v_{M+1}) = v_2v_3\cdots v_{M+1}.
$$
If we let $G = (V, E, t, i)$, then it is clear that $\Gamma^{[M]} = \Lambda(G)$ from \eqref{edgeshift}. Hence by Proposition \ref{M_Higher_Block_Code}, $\Gamma$ is topologically conjugate to an edge shift.
\end{proof}

Note that Propositions \ref{higher_block_conjugacy} and \ref{edge_shift_conjugacy} say that $M$-step shifts can always be recoded as $L$-step shifts for any $L<M$ and as an edge shift in the same way that occurs in the classical theory for shift spaces over finite alphabets. These results establish a remarkable difference with respect to the theory of  Ott-Tomforde-Willis one-sided shift spaces \cite{OTW14}. In fact, if one uses the definition of sliding block codes given in \cite[Definition 7.1]{OTW14}, then it is only possible to prove a version of Proposition \ref{M_Higher_Block_Code} if $\Lambda$ is a row-finite shift (see \cite[Proposition 6.5]{OTW14}). In such a case Propositions \ref{higher_block_conjugacy} and \ref{edge_shift_conjugacy} will also hold only for row-finite shifts (see \cite[Propositions 6.2 and 6.6]{OTW14}). In fact, \cite[Example 5.18]{OTW14} presents an example of non row-finite Ott-Tomforde-Willis one-sided 1-step shift spaces that are not conjugate to edge shifts, while \cite{GR} presents a non row-finite Ott-Tomforde-Willis one-sided $(M+1)$-step shift space that is not conjugate to any $M$-step shift (both using \cite[Definition 4.1]{OTW14} for conjugacy which imposes that the length of sequences shall be preserved).

\section{Relationships between two-sided and one-sided shift spaces}\label{SEC:2-sided_1-sided}

In this section we flesh out the relationship between two-sided shift space and Ott-Tomforde-Willis one-sided shift spaces hinted at in Remark \ref{inverse_limit1} and Proposition \ref{k-projection}. We begin by showing that shift spaces can be given by a set of words that is {\em minimal} in a suitable sense.

\begin{lem}\label{LEMMA:Forbidden_words} Let $\mathbf{F}\subset A^{\Z^-}\cup\bigcup_{k\geq 1}A^k$ and let $X_{\mathbf{F}}^{\text{inf}}\subset A^\Z$ be as in \eqref{forbidden_words_2-sided}. If $w\in \bigcup_{k\geq 1}A^k$ does not belong to $B(X_{\mathbf{F}}^{\text{inf}})$, then there exists a subblock $u$ of $w$ such that:
\begin{enumerate}
\item $u\notin B(X_{\mathbf{F}}^{\text{inf}})$;
\item We have that $v\in B(X_{\mathbf{F}}^{\text{inf}})$ for all proper subblocks $v$ of $u$; and
\item $u$ is a subblock of some word in $\mathbf{F}$.
\end{enumerate}
\end{lem}

\begin{proof}

Recall that $X_{\mathbf{F}}^{\text{inf}}$ consists of all elements in $A^\Z$ which do not contain a subblock which is an element of $\mathbf{F}$.
Suppose that $w$ is a finite block which is not
in $B(X_{\mathbf{F}}^{\text{inf}})$. Let $u$ be a subblock of $w$ such that $u$ is not in $B(X_{\mathbf{F}}^{\text{inf}})$ and such that all proper subwords of $u$ belong to $B(X_{\mathbf{F}}^{\text{inf}})$ (note that $u$ is not necessarily unique, and it is possible that $u=w$).

If $u$ is a single letter, say $u=(a)$, then the letter $a$ does not appear in any element of $X_{\mathbf{F}}^{\text{inf}}$. But the only way that this can happen is if there is a word in $\mathbf{F}$ which contains $a$, and we are done.


So, suppose $u$ is not a single letter, and let $M>1$ be the length of $u$. Let $\mathbf{G}:=\mathbf{G}_1\cup\mathbf{G}_2$, where $$\mathbf{G}_1:=\{v\in A^M:\ (v_1\ldots v_k)\in\mathbf{F},\text{ for some } 1\leq k< M\}\qquad
\text{and}\qquad
\mathbf{G}_2:=\mathbf{F}\cap\bigcup_{k\geq M}A^k.$$

It is direct that $X_{\mathbf{G}}^{\text{inf}}=X_{\mathbf{F}}^{\text{inf}}$. Furthermore, $u$ is not a prefix of any word in $\mathbf{G}_1$, because that would imply existence of a proper subword of $u$ in $\mathbf{F}$, contradicting that any proper subword of $u$ belongs to $B(X_{\mathbf{F}}^{\text{inf}})$.
Since every element of $\mathbf{G}$ has length greater than or equal to $M$, it follows from Corollary \ref{COR:M_Higher_Block_Code} that ${X_{\mathbf{F}}^{\text{inf}}}^{[M]}:=\Xi^{[M]}(X_{\mathbf{F}}^{\text{inf}})=X_{\mathbf{G}^{[M]}}^{\text{inf}}$.

Now $[u]:=\Xi^{[M]}(u)$ is a letter which does not appear in any element of ${X_{\mathbf{F}}^{\text{inf}}}^{[M]}$, so as before $[u]$ is a subword of some word of $\mathbf{G}^{[M]}$ (possibly $[u]$ is itself in $\mathbf{G}^{[M]}$). But this means that $u$ is a subblock of some element of $\mathbf{G}$. Since $\mathbf{G}_1$ does not contain any block with $u$ as subblock, $u$ appears as subword of at least one element of $\mathbf{G}_2\subset\mathbf{F}$.

\end{proof}

Note that the above lemma is also true if we replace $X_{\mathbf{F}}^{\text{inf}}$ with the set of the infinite-length elements of a Ott-Tomforde-Willis one-sided shift space (defining $\Xi^{[M]}$ conveniently) and even if we replace it with a classical shift space over a finite alphabet.

We now define the notion of a minimal set of forbidden words and show that, for shifts given by a minimal set of forbidden words, the language of the set of infinite sequences in the shift does not depend on the left infinite forbidden words (that is, does not depend on left infinite restrictions).

\begin{defn}\label{minimal}
Let $\mathbf{F}\subset A^{\Z^-}\cup\bigcup_{k\geq 1}A^k$. We say that $\mathbf{F}$ is {\em minimal} if whenever $w\in\mathbf{F}$ and $u\in \bigcup_{k\geq 1}A^k$ is a proper finite subblock of $w$, then $u\in B(X_{\mathbf{F}})$.
\end{defn}


We now show that every shift space can be defined by a minimal set of forbidden words.

\begin{prop} Let $\Lambda$ be a shift space. Then we can find a minimal set $\mathbf{F}$ such that $\Lambda=X_{\mathbf{F}}$.
\end{prop}

\begin{proof}By Proposition \ref{shiftforbiddenwords}, we can find $\mathbf{G}\subset A^{\Z^-}\cup\bigcup_{k\geq 1}A^k$ such that $\Lambda=X_{\mathbf{G}}$. We now apply Lemma \ref{LEMMA:Forbidden_words} and define $\mathbf{F}$ as the set of all words $u\in A^{\Z^-}\cup\bigcup_{k\geq 1}A^k$ such that either $u\in\mathbf{G}$ and all finite proper subblocks of $u$ belong to $B(\Lambda)$, or
$u$ is a finite proper subblock of some $w\in \mathbf{G}$ and all proper subblocks of $u$ are in $B(\Lambda)$.
\end{proof}

We note that if $\Lambda=X_{\mathbf{F}}$ is a shift space and $\mathbf{F}$ is minimal, then it is not possible to replace restrictions given by left-infinite sequences by restrictions given only by finite words.

We now give examples and nonexamples of minimal sets.

\begin{ex}\label{minimal_examples}

In the following examples, we write our set of forbidden words $\mathbf{F}$ as a union $\mathbf{F}=\mathbf{F}'\cup\mathbf{F}''$, with $\mathbf{F}'\subset \bigcup_{k\geq 1}A^k$ and $\mathbf{F}''\subset A^{\Z^-}$.

\begin{enumerate}
\item[a)] Let $A=\N$, let $\mathbf{F}'=\emptyset$, and let $$\mathbf{F}''=\{(x_i)_{i\leq 0}\in  A^{\Z^-}: \inf\{i:\ x_i\neq 1\}=-\infty\}.$$ Then $\Lambda:=X_\mathbf{F}$ is a two-sided shift space with $B(\Lambda)=B(\Sigma_A^\Z)$ such that, for any $x\in\Lambda$, there exists $k\in\Z$ such that $x_i=1$ for all $i\leq k$. Clearly $\mathbf{F}$ is minimal and, furthermore, it is not possible to replace the restrictions given by $\mathbf{F}''$ by restrictions given only by finite words.

\item[b)] Take $k\in\N$, let $A=\N$, let $\mathbf{F}'=\{11\}$, let $$\mathbf{F}''=\{(x_i)_{i\leq 0}\in  A^{\Z^-}:\ x_{i-k}\neq x_i,\ \forall i\leq 0\},$$ and let $\Lambda:=X_\mathbf{F}$. Then  $\Lambda^{\text{fin}}=\{\O\}$ and $$\Lambda^{\text{inf}}=\{x\in \Sigma_A^\Z: \ x\ has\ period\ k,\ and\ x_ix_{i+1}\neq 11\ \forall i\in\Z\}.$$ It follows that $\mathbf{F}$ is not minimal, since any word of the form $u=u_1\ldots u_{k+1}$ with $u_{1}\neq u_{k+1}$ is a subblock of some sequence of $\mathbf{F}''$, but $u\notin B(\Lambda)$. However,
the set $$\mathbf{G}:=\{11,\ u_1\ldots u_{k+1}:\ u_{1}\neq u_{k+1}\},$$ is minimal and $\Lambda=X_{\mathbf{G}}$.

\item[c)] Let $A=\N$, let $\mathbf{F}'=\emptyset$, let  $$\mathbf{F}''=\{(x_i)_{i\leq 0}\in  A^{\Z^-}:\ \exists k\geq 1,\ x_{i-k}\neq x_i,\ \forall i\leq 0\},$$ and let $\Lambda:=X_\mathbf{F}$. Then $\Lambda^{\text{fin}}=\{\O\}$ and $\Lambda^{\text{inf}}=\{x\in \Sigma_A^\Z: \ x\ is\ periodic\}$. Since $B(\Lambda)=B(\Sigma_A^\Z)$, $\mathbf{F}$ must be minimal.

\item[d)] Let $A=\N$, let
$$\mathbf{F}'=\{n2: \ n\in A\},$$
$$\mathbf{F}''=\{(x_i)_{i\leq 0}\in  A^{\Z^-}:\ \exists i\leq 0,\ x_i=1\}.$$ Here, $\mathbf{F}$ is not minimal. In fact, 2 is a proper subblock of all words in $\mathbf{F}'$ and it does not belong to $B(X_{\mathbf{F}})$, and, on the other hand, 1 is a proper subblock of all words in $\mathbf{F}''$ and it does not belong to $B(X_{\mathbf{F}})$. Here, the set $\mathbf{G}=\{1,2\}$ is a minimal set such that $X_{\mathbf{F}}=X_{\mathbf{G}}=\Sigma_{A\setminus\{1,2\}}^\Z$.
\end{enumerate}
\end{ex}


We now show that if $\mathbf{F}$ is minimal, then the language of $X_\mathbf{F}$ is not affected by the left-infinite words in $\mathbf{F}$.
\begin{prop}\label{minimality-language} Let $\mathbf{F}\subset A^{\Z^-}\cup\bigcup_{k\geq 1}A^k$ be a minimal set, where $\mathbf{F}=\mathbf{F}'\cup\mathbf{F}''$ with $\mathbf{F}'\subset\bigcup_{k\geq 1}A^k$ and $\mathbf{F}''\subset A^{\Z^-}$. Then $B(X_{\mathbf{F}}^{\text{inf}})=B(X_{\mathbf{F}'}^{\text{inf}})$, that is, the restrictions imposed by $\mathbf{F}''$ do not affect
the language of $X_{\mathbf{F}}^{\text{inf}}$.
\end{prop}

\begin{proof}
Since $\mathbf{F}'\subset \mathbf{F}$, it is straighforward that $B(X_{\mathbf{F}}^{\text{inf}})\subset B(X_{\mathbf{F}'}^{\text{inf}})$. To prove that $B(X_{\mathbf{F}'}^{\text{inf}})\subset B(X_{\mathbf{F}}^{\text{inf}})$, let $w\in B(X_{\mathbf{F}'}^{\text{inf}})$ and suppose by contradiction that $w\notin B(X_{\mathbf{F}}^{\text{inf}})$.  Lemma \ref{LEMMA:Forbidden_words} implies that there exists a subword $u$ of $w$ such that $u\notin B(X_{\mathbf{F}}^{\text{inf}})$ and $u$ is a subword of some word in $\mathbf{F}$. However, since $\mathbf{F}$ is minimal we have that $u$ is in $\mathbf{F}$, and since $u$ is finite, then $u\in \mathbf{F}'$ which contradicts the fact that
$w\in B(X_{\mathbf{F}'}^{\text{inf}})$.
Hence $B(X_{\mathbf{F}'}^{\text{inf}})\subset B(X_{\mathbf{F}}^{\text{inf}})$.

\end{proof}

The following result gives a relationship between two-sided shift spaces defined in this work and  Ott-Tomforde-Willis one-sided shift spaces.

\begin{prop}\label{projection} Let $\Lambda\subseteq\Sigma_A^\Z$ be a shift space such that $\Lambda=X_\mathbf{F}$, where $\mathbf{F}$ is minimal. Let $\mathbf{F}=\mathbf{F'}\cup\mathbf{F}''$ with $\mathbf{F}'\subset\bigcup_{k\geq 1}A^k$, $\mathbf{F}''\subset A^{\Z^-}$ and let $\pi:\Lambda\to\Sigma_A^\N$ be the projection onto the positive coordinates.
\begin{enumerate}
\item[(i)] If $|\LA|=\infty$ then $\pi(\Lambda)$ is dense in the one-sided Ott-Tomforde-Willis shift space $\hat X_{\mathbf{F}'}\subset \Sigma_A^\N$;

\item[(ii)] If $|\LA|<\infty$ then $\pi(\Lambda\setminus\{\O\})$ is the standard shift space of $\hat X_{\mathbf{F}'}\subset (\LA)^\N$.
\end{enumerate}
\end{prop}

\begin{proof} {\color{white}.}

\begin{enumerate}
\item[(i)] We will prove that $\pi(\Lambda^{\text{inf}})$ is dense in $\hat X_{\mathbf{F}'}^{\text{inf}}$ and that any element of $\pi(\Lambda^{\text{fin}})$ is a finite sequence which satisfies the infinite extension property in $\hat X_{\mathbf{F}'}$. Hence we will have
    $$\overline{\pi(\Lambda)}=\overline{\pi(\Lambda^{\text{inf}})}=\overline{\hat X_{\mathbf{F}'}^{\text{inf}}}=\hat X_{\mathbf{F}'}.$$

For this, first note that $$B(\pi(\Lambda^{\text{inf}}))=B(\Lambda^{\text{inf}})=B(X_{\mathbf{F}'}^{\text{inf}})=B(\hat X_{\mathbf{F}'}^{\text{inf}}),$$
where the first equality is due to $\s(\Lambda)=\Lambda$, the second equality is due to Proposition \ref{minimality-language} and the third equality is due the definitions of $X_{\mathbf{F}'}^{\text{inf}}$ and $\hat X_{\mathbf{F}'}^{\text{inf}}$.

The above means that $\pi(\Lambda^{\text{inf}})\subset \hat X_{\mathbf{F}'}^{\text{inf}}$. Now, take $x\in \hat X_{\mathbf{F}'}^{\text{inf}}$,
and find $y$ such that $x\in Z(y)$. It follows that $y\in B(\hat X_{\mathbf{F}'}^{\text{inf}})=B(\pi(\Lambda^{\text{inf}}))$ and, therefore, there exists $z\in\pi(\Lambda^{\text{inf}})$ such that $z_i=x_i=y_i$ for all $1\leq i\leq l(y)$. Hence $z\in Z(y)\cap \pi(\Lambda^{\text{inf}})$, proving that $\pi(\Lambda^{\text{inf}})$ is dense in $\hat X_{\mathbf{F}'}^{\text{inf}}$.

Now, observe that $\pi(x)=\O$ whenever $l(x)\leq 0$, and $\pi$ is continuous at $x\in \Lambda^{\text{fin}}$ whenever $l(x)\geq 0$ by Proposition \ref{k-projection}. The last statment implies that any finite sequence $\pi(x)\in\pi(\Lambda)$ will have the infinite extension property.

\item[(ii)] In this case $\Lambda^{\text{fin}}\subset\{\O\}$. Therefore,  $\Lambda^{\text{inf}}=\Lambda\setminus\{\O\}$ and  the result follows from the fact that $\pi(\Lambda^{\text{inf}})$ is dense in $\hat X_{\mathbf{F}'}^{\text{inf}}$.
\end{enumerate}

\end{proof}

We apply Proposition \ref{projection} to each of the examples from Example \ref{minimal_examples}.

\begin{ex}\label{projection_examples}

\begin{enumerate}
\item[a)] Let $\Lambda=X_{\mathbf{F}}$, where $\mathbf{F}$ is the set given in Example \ref{minimal_examples}.a. We have that $\LA=A$ and $\pi(\Lambda)=\hat X_{\mathbf{F}'}=\Sigma_A^\N$.

\item[b)] Let $\mathbf{F}$ and $\mathbf{G}$ be as in Example \ref{minimal_examples}.b, and let $\Lambda=X_{\mathbf{F}}$. Then $\LA=A$ and $\pi(\Lambda)$ is the set consisting of $\O$ together with all infinite-length elements of $\Sigma_A^\N$ with period $k$ which do not contain the word 11. Hence, the closure of $\pi(\Lambda)$ is $\hat X_{\mathbf{F}'\cup \mathbf{G}}\subset \Sigma_A^\N$ which, aside from the elements of $\pi(\Lambda)$, contains all finite sequences with period $k$ where the word 11 does not appear.

\item[c)] Let $\Lambda=X_{\mathbf{F}}$, where $\mathbf{F}$ is the set given in Example \ref{minimal_examples}.c. Then $\LA=A$, and $\pi(\Lambda)=\{\O\}\cup\{x\in\Sigma_A^\N:\ x\ is\ periodic\}$. The closure of $\pi(\Lambda)$ is $\hat X_{\mathbf{F}'}=\Sigma_A^\N$.

\item[d)] Let $\Lambda=X_{\mathbf{F}}$, where $\mathbf{F}$ is the set given in Example \ref{minimal_examples}.d. Here, $\LA=A\setminus\{1,2\}$ and $\pi(\Lambda)=\hat X_{\mathbf{F}'}=\Sigma_A^\N$.
\end{enumerate}

\end{ex}

Proposition \ref{projection} gives us a way of obtaining an Ott-Tomforde-Willis one-sided shift space from a two-sided shift space. We now use Remark \ref{inverse_limit1} to construct a two-sided shift space from a one-sided shift space, although in general these two operations are not inverses of each other, see Remarks \ref{contraexemplo1} and \ref{contraexemplo2}.

\begin{prop}\label{inverse_limit2}
Let $\mathbf{F}\subset\bigcup_{k\in\N}A^k$ be a minimal set and consider Ott-Tomforde-Willis one-sided shift $\hat X_{\mathbf{F}}$. Consider the set obtained by taking the inverse limit of $\hat X_{\mathbf{F}}$,  $$(\hat X_{\mathbf{F}})^\s:=\{(\mathcal{X}_i)_{i\in\Z}:\ \forall i\in\Z\ \mathcal{X}_i\in\hat X_{\mathbf{F}}\text{ and }\s(\mathcal{X}_i)=\mathcal{X}_{i+1}\}$$ and define $\Lambda\in\Sigma_A^\Z$ by $$\Lambda:=p\big((\hat X_{\mathbf{F}})^\s\big)=\{(x_i)_{i\in\Z}\in\Sigma_A^\Z:\ (x_{i+j-1})_{j\in\N}\in\Lambda\ \forall i\in\Z\},$$ where $p$ is the map defined in Remark \ref{inverse_limit1}.

\begin{enumerate}
\item[(i)] If either $|L_{\hat X_{\mathbf{F}}}|=\infty$ or $|\hat X_{\mathbf{F}}|<\infty$, then $\Lambda=X_{\mathbf{F}}$;

\item[(ii)] If $|L_{\hat X_{\mathbf{F}}}|<\infty$ and $|\hat X_{\mathbf{F}}|=\infty$, then $\Lambda\cup\{\O\}=X_{\mathbf{F}}$.

\end{enumerate}

\end{prop}

\begin{proof}
First, let us prove that in both cases we have $\Lambda^{\text{inf}}= X_{\mathbf{F}}^{\text{inf}}$. For this, observe that $\Lambda^{\text{inf}}=p\big((\Lambda'^{\text{inf}})^\s\big)$. From the definition, elements of $\Lambda^{\text{inf}}$ do not contain any elements of $\mathbf{F}$ as subblocks, which means that $\Lambda^{\text{inf}}\subset X_{\mathbf{F}}^{\text{inf}}$. On the other hand, by Proposition \ref{projection} it follows that $\pi(X_{\mathbf{F}}^{\text{inf}})\subset\hat X_{\mathbf{F}}^{\text{inf}}$, which means that given $(x_i)_{i\in\Z}\in X_{\mathbf{F}}^{\text{inf}}$, we have that $\mathcal{X}_i=(x_{i+j-1})_{j\in\N}\in\hat X_{\mathbf{F}}^{\text{inf}}$ for all $i\in\Z$. Hence we get that $(x_i)_{i\in\Z}\in\Lambda$ and therefore $X_{\mathbf{F}}^{\text{inf}}\subset \Lambda^{\text{inf}}$.

Before we proceed, note that $\LA=L_{X_{\mathbf{F}}}=L_{\hat X_{\mathbf{F}}}$. In fact, since $\Lambda^{\text{inf}}= X_{\mathbf{F}}^{\text{inf}}$, it is straighforward that $\LA=L_{X_{\mathbf{F}}}$. Furthermore, from the definition of $\Lambda$, it follows that $\LA\subset L_{\hat X_{\mathbf{F}}}$. On the other hand if by contradiction we suppose that $a\in \LAA\setminus\LA$, then there does not exist an element in $\Lambda^{\text{inf}}$ using the letter $a$. But $\Lambda^{\text{inf}}= X_{\mathbf{F}}^{\text{inf}}$ and, since $\mathbf{F}$ is minimal, then necessarily $a\in \mathbf{F}$, which contradicts that $a\in L_{\hat X_{\mathbf{F}}}$.

Now, to complete the proof, we have two cases:

\begin{enumerate}
\item[(i)] If $|L_{\hat X_{\mathbf{F}}}|=\infty$ then the one-sided empty sequence belongs to $\hat X_{\mathbf{F}}$ and therefore the two-sided empty sequence belongs to $\Lambda$. Since $L_{X_{\mathbf{F}}}=L_{\hat X_{\mathbf{F}}}$ this implies that the two-sided empty sequence belongs to $X_{\mathbf{F}}$.

As before, we have that $\Lambda^{\text{fin}}\subset X_{\mathbf{F}}^{\text{fin}}$. On the other hand, to check that $X_{\mathbf{F}}^{\text{fin}}\subset \Lambda^{\text{fin}}$ we take $x\in X_{\mathbf{F}}^{\text{fin}}\setminus\{\O\}$ and use the same argument used for infinite sequences.

    If $|\hat X_{\mathbf{F}}|<\infty$ then $|L_{\hat X_{\mathbf{F}}}|<\infty$ and hence $\hat X_{\mathbf{F}}'$ consists of a finite number of periodic infinite sequences, which implies that $\hat X_{\mathbf{F}}^{\text{fin}}$ is empty. Therefore $\Lambda$ consists only of a finite number of periodic infinite sequences and hence $\Lambda^{\text{fin}}$ is empty. Since $\Lambda^{\text{inf}}= X_{\mathbf{F}}^{\text{inf}}$, we have that $X_{\mathbf{F}}^{\text{inf}}$ contains just a finite number of periodic infinite sequences, concluding that $X_{\mathbf{F}}^{\text{fin}}$ is also empty.

\item[(ii)] Suppose $|L_{\hat X_{\mathbf{F}}}|<\infty$ and $|\hat X_{\mathbf{F}}|=\infty$. Since $|L_{\hat X_{\mathbf{F}}}|<\infty$, it follows that $\hat X_{\mathbf{F}}$ is a classical shift space over a finite alphabet, which implies that $\hat X_{\mathbf{F}}^{\text{fin}}$ is empty. Hence $\Lambda^{\text{inf}}= X_{\mathbf{F}}^{\text{inf}}$ and contains an infinite number of infinite sequences over a finite alphabet. Since $\Lambda^{\text{fin}}$ is empty and $X_{\mathbf{F}}^{\text{fin}}=\{\O\}$, the result follows.

\end{enumerate}

\end{proof}

So using Proposition \ref{inverse_limit2} we can produce a two-sided shift space from an Ott-Tomforde-Willis one-sided shift space, and using Proposition \ref{projection} we can produce a one-sided shift space by projecting a two-sided shift space. The following two remarks show that these two operations are not necessarily the inverses of each other.
\begin{rmk}\label{contraexemplo1} Given a shift space $\Lambda\subseteq\Sigma_A^\Z$ with $|\LA|=\infty$, take $\Lambda'$ as the closure of $\pi(\Lambda)$ in $\Sigma_A^\N$. Then, in general, $p\big((\Lambda')^\s\big)\neq\Lambda$. To see this, suppose that  $\Lambda=X_{\mathbf{F}}$, where $\mathbf{F}={\mathbf{F}'\cup\mathbf{F}''}$ is minimal and $\mathbf{F}''\subset A^{\Z^-}$ nonempty. By Proposition \ref{inverse_limit2}.i $\Lambda'=\hat X_{\mathbf{F}'}$, so we have $p\big((\Lambda')^\s\big)= X_{\mathbf{F}'}$ and $\Lambda\subsetneq X_{\mathbf{F}'}$.
\end{rmk}

\begin{rmk}\label{contraexemplo2}
Given a one-sided shift space $\Lambda'\subset\Sigma_A^\N$, in general $\pi\Big(p\big((\Lambda')^\s\big)\Big)\neq\Lambda'$. For instance, take $A=\N$ and the non minimal set $\mathbf{F}=\{n1:\ n\in A\}$, so that $\Lambda':=\hat X_{\mathbf{F}}$ is the one-sided shift space of all sequences where 1 does not appear, except maybe as the first symbol. Hence, $p\big((\Lambda')^\s\big)$ is the two-sided full shift over the alphabet $B:=A\setminus\{1\}$ and, therefore, $\pi\Big(p\big((\Lambda')^\s\big)\Big)=\Sigma_B^\N\subsetneq\Lambda'$.

\end{rmk}

\section{Final discussion}\label{SEC:final}

In this work we have proposed a definition of two-sided shift spaces over countably infinite alphabets. The topology we proposed can be viewed as a two-sided version of the the space proposed by Ott-Tomforde-Willis \cite{OTW14} for one-sided shift spaces over infinite alphabets.

Although the set of sequences proposed here as the two-sided shift space can be viewed as the inverse-limit system of an Ott-Tomforde-Willis one-sided shift space, we remark that the topology considered was not the product topology (which is the usual topology for inverse-limit systems). We made this choice because the product topology does not have a basis of clopen sets. Thus, instead of the product topology, we considered a topology with a basis consisting of sets defined by specifying infinitely many coordinates and by the complement of such sets.

To contrast the two, notice that the product topology contains sets of the form $$\{x\in\Sigma_A^\Z:\ x_i=a_i, \forall \ i\in I\}$$ and  $$\{x\in\Sigma_A^\Z:\ x_i\neq a_i, \forall \ i\in I\}$$ for any choice of finite $I\subset\Z$ and $\{a_i\}_{i\in I}\subset A$. On the other hand, the topology we considered in this work contains only sets of the first type.

The key features to note about the shift spaces we introduced here are that they are zero-dimensional, Hausdorff, compact and sequentially compact spaces. Furthermore, the shift map is continuous with respect to this topology (while it was not continuous for Ott-Tomforde-Willis one-sided shift spaces). These features were key for us to generalize work of Kitchens \cite{kitchens} to inverse semigroup shifts over infinite alphabets \cite{GSS1}. It remains open whether the locally compact space obtained by excluding the empty sequence is metrizable.

We also note that because our clopen basis keeps track of an infinite number of entries, the dynamical recurrence properties of the shift map which hold for finite-alphabet shift spaces and for Ott-Tomforde-Willis one-sided shift spaces do not hold in our case. Indeed, in those situations the full shift is topologically chaotic (the set of its periodic orbits is dense, and it is topologically transitive - see \cite{banks92}) while in our case the empty sequence in the full two-sided shift is an attractor for all nonperiodic orbits. However, we remark that this apparent simplicity of the global dynamics hides the complexity of the local behavior. This can be visualized by taking the projection of the two-sided shift on the positive coordinates, which will result in a dense subset of the Ott-Tomforde-Willis one-sided shift space (which is topologically chaotic).

Interestingly, the problem of finding general conditions to characterize the existence of an analogue of the Curtis-Hedlund-Lyndon Theorem remains open. In fact, the class of maps for which we could expect to find a complete characterization in terms of a  Curtis-Hedlund-Lyndon type theorem should be compatible with the topology considered, that is, it should contain all maps $\Phi:\Lambda\to\Lambda$ for which
for all $n\in\Z$ and $x\in\Lambda$ there exist $\ell\geq 0$ which does not depend on the value of $n$, but does depend on the configuration of $x$ on the coordinates $(\ldots x_{n-1}x_n \ldots x_{n+\ell})$ to determine $\bigl(\Phi(x)\bigr)_n$.



\section*{Acknowledgments}

\noindent D. Gon\c{c}alves was partially supported by Capes grant PVE085/2012 and CNPq.

\noindent
M. Sobottka was supported by CNPq-Brazil grants 304813/2012-5, 480314/2013-6 and 308575/2015-6.. Part of this work was
carried out while the author was postdoctoral fellow of CAPES-Brazil at Center for Mathematical Modeling, University of Chile.

\noindent C. Starling was supported by CNPq, and work on this paper occurred while the author held a postdoctoral fellowship at UFSC.


\end{document}